\documentclass[11pt,onecolumn]{article}

\usepackage{stackengine}
\usepackage{a4wide}
\usepackage{amsmath,amsthm,epsfig,amssymb,amsbsy}
\usepackage{enumerate}
\usepackage{comment}
\usepackage{algorithm}
\usepackage{algorithmic}
\usepackage{amsfonts}       
\usepackage{threeparttable}
\usepackage{dsfont}
\usepackage{enumitem}
\usepackage{amssymb}
\usepackage{tikz}
\usepackage{tikz-3dplot}
\usepackage{pgfplots}
\usetikzlibrary{backgrounds}
\usepackage{makecell}

\usepackage[
giveninits=true,
maxbibnames=9,
maxcitenames=2,
backend=biber,
bibstyle=numeric,
doi=false,isbn=false,url=false,
]{biblatex}
\renewbibmacro{in:}{
    \ifentrytype{article}
    {}
    {\printtext{\bibstring{in}\intitlepunct}}
    }
\newenvironment{keywords}{\begin{paragraph}{Keywords:}
}
{
\end{paragraph}
}
    
\newenvironment{subclass}{\begin{paragraph}{AMS Subject Classification:}
}
{\end{paragraph}
}

\DeclareMathOperator*{\argmin}{arg\,min}

\DeclareMathOperator{\gph}{gph}

\DeclareMathOperator{\dom}{dom}

\DeclareMathOperator{\intr}{int}
\DeclareMathOperator{\bdry}{bdry}

\DeclareMathOperator{\ran}{rge}

\DeclareMathOperator{\id}{id}

\newcommand{\bR}{\mathbb{R}}

\newcommand{\bN}{\mathbb{N}}

\newcommand{\exR}{\overline{\mathbb{R}}}

\DeclareMathOperator{\gap}{gap}
\newcommand{\gapf}{\gap_g^\Phi}
\newcommand{\gapg}{\gap_f^\Phi}
\newcommand{\valf}{\mathcal{V}_F}

\newcommand{\cC}{\mathcal{C}}

\newcommand{\cI}{\mathcal{I}}
\newcommand{\cJ}{\mathcal{J}}

\newcommand\xrowht[2][0]{\addstackgap[.5\dimexpr#2\relax]{\vphantom{#1}}}

\makeatletter
\newcommand{\prox}[3][\@nil]{%
  \def\tmp{#1}%
   \ifx\tmp\@nnil
       \operatorname{prox}_{#3}^{#2}
    \else
         \operatorname{prox}_{#3}^{#1 \star #2}
    \fi}

\newcommand{\bprox}[3][\@nil]{%
  \def\tmp{#1}%
   \ifx\tmp\@nnil
       \operatorname{bprox}_{#3}^{#2}
    \else
        \operatorname{bprox}_{#3}^{#1 #2}
    \fi}
\makeatother

\newcommand{\phif}{\nabla_\Phi f}
\newcommand{\phibreg}{D_f^\Phi}

\usepackage{hyperref}
\hypersetup{
    colorlinks=true,
    linkcolor=blue,
    filecolor=magenta,
    citecolor =magenta,  
    urlcolor=magenta,
    pdftitle={Generalized DCA}
    }
\usepackage[capitalize]{cleveref}[0.19]

\crefname{section}{section}{sections}
\crefname{subsection}{subsection}{subsections}
\Crefname{section}{Section}{Sections}
\Crefname{subsection}{Subsection}{Subsections}

\Crefname{figure}{Figure}{Figures}

\crefformat{equation}{\textup{#2(#1)#3}}
\crefrangeformat{equation}{\textup{#3(#1)#4--#5(#2)#6}}
\crefmultiformat{equation}{\textup{#2(#1)#3}}{ and \textup{#2(#1)#3}}
{, \textup{#2(#1)#3}}{, and \textup{#2(#1)#3}}
\crefrangemultiformat{equation}{\textup{#3(#1)#4--#5(#2)#6}}%
{ and \textup{#3(#1)#4--#5(#2)#6}}{, \textup{#3(#1)#4--#5(#2)#6}}{, and \textup{#3(#1)#4--#5(#2)#6}}

\Crefformat{equation}{#2Equation~\textup{(#1)}#3}
\Crefrangeformat{equation}{Equations~\textup{#3(#1)#4--#5(#2)#6}}
\Crefmultiformat{equation}{Equations~\textup{#2(#1)#3}}{ and \textup{#2(#1)#3}}
{, \textup{#2(#1)#3}}{, and \textup{#2(#1)#3}}
\Crefrangemultiformat{equation}{Equations~\textup{#3(#1)#4--#5(#2)#6}}%
{ and \textup{#3(#1)#4--#5(#2)#6}}{, \textup{#3(#1)#4--#5(#2)#6}}{, and \textup{#3(#1)#4--#5(#2)#6}}

\newtheorem{theorem}{Theorem}[section]
\newlist{thmenum}{enumerate}{1} 
\setlist[thmenum]{label=(\roman*), ref=\theproposition(\roman*), font=\rm}
\crefalias{thmenumi}{theorem} 

\newtheorem{corollary}[theorem]{Corollary}
\newtheorem{lemma}[theorem]{Lemma}
\newlist{lemenum}{enumerate}{1} 
\setlist[lemenum]{label=(\roman*), ref=\theproposition(\roman*), font=\rm}
\crefalias{lemenumi}{lemma} 

\newtheorem{proposition}[theorem]{Proposition}
\newlist{propenum}{enumerate}{1} 
\setlist[propenum]{label=(\roman*), ref=\theproposition(\roman*), font=\rm}
\crefalias{propenumi}{proposition} 

\newlist{defenum}{enumerate}{1} 
\setlist[defenum]{label=(\roman*), ref=\thedefinition(\roman*), font=\rm}
\crefalias{defenumi}{definition} 

\newlist{corenum}{enumerate}{1} 
\setlist[corenum]{label=(\roman*), ref=\thedefinition(\roman*), font=\rm}
\crefalias{corenum}{corollary} 

\newtheorem{definition}[theorem]{Definition}

\newtheoremstyle{boldremark}
    {\dimexpr\topsep/2\relax} 
    {\dimexpr\topsep/2\relax} 
    {}          
    {}          
    {\bfseries} 
    {.}         
    {.5em}      
    {}          

\theoremstyle{boldremark}
\newtheorem{remark}[theorem]{Remark}
\theoremstyle{boldremark}
\newtheorem{example}[theorem]{Example}

\newlist{assumenum}{enumerate}{1} 
\setlist[assumenum]{leftmargin=2.1cm,label=(A\arabic*),font=\bfseries}
\creflabelformat{assumenumi}{#2#1#3}
\crefname{assumenumi}{assumption}{assumptions}
\Crefname{assumenumi}{Assumption}{Assumptions}

\usepackage{xcolor}
\newcommand{\change}[2]{#2}

\title{Forward-backward splitting
\\
under the light of generalized convexity}
\author{Konstantinos Oikonomidis\footnotemark[1]\thanks{KU Leuven,
		Department of Electrical Engineering (ESAT-STADIUS),
		Kasteelpark Arenberg 10, 3001 Leuven, Belgium~
		{\tt%
            \href{mailto:konstantinos.oikonomidis@esat.kuleuven.be}{\{konstantinos.oikonomidis,}%
			\href{mailto:panos.patrinos@esat.kuleuven.be}{panos.patrinos\}}%
			\href{mailto:konstantinos.oikonomidis@esat.kuleuven.be,panos.patrinos@esat.kuleuven.be}{@esat.kuleuven.be}%
		}
	} \and Emanuel Laude\footnotemark[2]\thanks{Proxima Fusion,
		Fl\"o{\ss}ergasse 2,
		81369 Munich, Germany~
		{\tt
        \href{mailto:elaude@proximafusion.com}{elaude@proximafusion.com}%
		}
	} \and Panagiotis Patrinos\footnotemark[1]}

\begin{document}

\maketitle
\begin{abstract}
In this paper we present a unifying framework for continuous optimization methods grounded in the concept of generalized convexity. Utilizing the powerful theory of $\Phi$-convexity, we propose a conceptual algorithm that extends the classical difference-of-convex method, encompassing a broad spectrum of optimization algorithms. Relying exclusively on the tools of generalized convexity we develop a gap function analysis that strictly characterizes the decrease of the function values, leading to simplified and unified convergence results. As an outcome of this analysis, we naturally obtain a generalized PL inequality which ensures $q$-linear convergence rates of the proposed method, incorporating various well-established conditions from the existing literature. Moreover we propose a $\Phi$-Bregman proximal point interpretation of the scheme that allows us to capture conditions that lead to sublinear rates under convexity.
\end{abstract}
\begin{keywords}
duality $\cdot$ generalized convexity $\cdot$ DC-programming
\end{keywords}
\begin{subclass}
65K05 $\cdot$ 49J52 $\cdot$ 90C30
\end{subclass}

\tableofcontents
\section{Introduction}
\subsection{Motivation}
A significant number of important optimization algorithms are generated by succesively minimizing upper bounds of the cost function, a principle often called majorization-minimization. This foundational approach underlies a wide array of methods, ranging from classical first-order methods such as the proximal gradient method \cite[Chapter 10]{beck2017first}, the Bregman proximal gradient method \cite{bolte2018first} and the difference-of-convex algorithm (DCA) \cite{horst1999dc}, to even higher-order methods such as the cubic regularization of Newton's method \cite{nesterov2006cubic}. It also extends to algorithms beyond traditional optimization, including Expectation Maximization and Sinkhorn's algorithm \cite{kunstner2021homeomorphic, aubin2022mirror, wang2024emparameterlearningframework}. 

While these algorithms are all rooted in the same majorization-minimization principle, their convergence analysis is often highly specific to the form of the upper bounds they employ. As a result, the analysis varies considerably across the literature, tailored to the unique characteristics of each method. This raises an intriguing question: Is it possible to create a unifying framework that not only generalizes the convergence analysis of these well-established methods but also facilitates the development of new algorithms without requiring significant modifications to the existing techniques?

In this paper we seek to address this question by leveraging the theory of $\Phi$-convexity, also known as $c$-concavity in the optimal transport literature. Originally introduced in an effort to transfer notions of convex analysis into nonlinear spaces (see \cite{bachir2017krein} and references therein), in the field of optimization theory $\Phi$-convexity has primarily been used in the context of eliminating duality gaps \cite{rockafellar1974augmented, penot1990strongly, bui2021zero}. Recently it has also been utilized in the context of learning \cite{blondel2022learning}. Following a line of work that explores optimization methods through the lens of of $\Phi$-convexity \cite{laude2023dualities, laude2022anisotropic}, we consider an extension of the classical difference-of-convex algorithm within the framework of generalized convexity. Parallel to our work, a unifying framework based on generalized convexity was also developed in \cite{léger2023gradient}, where an alternating minimization scheme was proposed. Our contribution is twofold:
\begin{enumerate}[label=(\roman*)]
    \item We study the difference of $\Phi$-convex problem, examining its optimality conditions and introducing the difference of $\Phi$-convex duality that generalizes the standard DC duality \cite{tao1997convex}. Building on the classical DCA, we introduce the $\Phi$-DCA and analyze its convergence properties using only the principles of (generalized) convexity: (generalized) subdifferentials and conjugates. We employ a novel approach that simultaneously considers both primal and dual regularized gap functions, enabling us to provide a strict characterization of function value decrease. We also introduce a generalized PL condition, which ensures the $q$-linear convergence rate of the proposed scheme. Moreover, we offer a $\Phi$-Bregman proximal point interpretation of our algorithm and establish sufficient conditions for achieving a sublinear rate of the function values in the convex regime.

    \item We demonstrate that numerous well-known algorithms can be viewed as specific instances of our generalized framework. We present examples that highlight the unifying properties of our proposed scheme across various aspects of analysis, including nonconvex subsequential convergence, sublinear rates under convexity, and $q$-linear convergence under the newly introduced generalized PL condition. Furthermore, our analysis leads to new sublinear rate guarantees for subclasses of the anisotropic proximal gradient algorithm \cite{laude2023dualities,laude2022anisotropic,oikonomidis2025nonlinearly}, while our linear convergence result \cref{lemma:apgm_linear:rate} improves upon the findings of \cite[Corollary 6.11]{laude2022anisotropic}.
\end{enumerate}

\subsection{Paper organization}
The remainder of this paper is organized as follows. \Cref{sec:gen_convexity} offers an overview of the notions of generalized convexity and conjugacy which provide valuable tools that are the building blocks of our analysis. In \cref{sec:prob_def} we describe the problem we study in this paper and present the difference of $\Phi$-convex duality as an extension of the classical difference of convex duality. In \cref{sec:alg} the proposed scheme is introduced and various examples of existent methods that can be considered instances of our generalized framework are provided. In \cref{sec:conv_analysis} the asymptotic convergence of the method is analyzed and the sublinear rate under convexity is studied, while in \cref{sec:linear_rate_phi_dca} a sufficient condition for $q$-linear convergence rates is stated and various examples are provided. Finally, \cref{sec:conclusion} concludes the paper.

\subsection{Notation and preliminaries}
We denote by $\langle \cdot, \cdot\rangle$ the standard Euclidean inner product on $\bR^n$ and by $\|\cdot\|:=\sqrt{\langle \cdot, \cdot \rangle}$ the standard Euclidean norm and the operator norm for matrices. The effective domain of an extended real-valued function $f : \bR^n \to \exR$ is denoted by $\dom f:=\{x\in\bR^n : f(x)<\infty\}$, and we say that $f$ is proper if $\dom f\neq\emptyset$ and $f(x) > -\infty$ for all $x \in \bR^n$; lower semicontinuous (lsc) if $f(\bar x)\leq\liminf_{x\to\bar x}f(x)$ for all $\bar x\in\bR^n$. We define by $\Gamma_0(\bR^n)$ the class of all proper, lsc convex functions $f:\bR^n \to \exR$, with $\Gamma_L(\bR^n)$ the ones that are $L$-weakly convex and with $\mathcal{C}^k(\bR^n)$ the ones which are $k$ times continuously differentiable. For a proper function $f :\bR^n \to \exR$ and $\lambda \geq 0$ we define the epi-scaling $(\lambda \star f)(x) = \lambda f(\lambda^{-1} x)$ for $\lambda > 0$ and $(\lambda \star f)(x)=\delta_{\{0\}}(x)$ otherwise. The set-valued mappings $\widehat{\partial}f, \partial f: \bR^n \rightrightarrows \bR^n$ are the regular and the limiting subdifferential of $f$, where $\bar v \in \widehat{\partial}f(\bar x)$ if $\liminf_{\substack{x\to\bar x \\ x \neq \bar x}} \frac{f(x)-f(\bar x) - \langle \bar v,x-\bar x\rangle}{\|x-\bar x\|} \geq 0$, while $\bar v \in \partial f(\bar x)$ if $\bar x \in \dom f$ and there exists a sequence $(x^k, v^k) \to (\bar x, \bar v)$, with $v^k \in \widehat{\partial}f(x^k)$ such that $f(x^k) \to f(\bar x)$. If $f \in \Gamma_0(\bR^n)$, the limiting subdifferential agrees with the one from convex analysis. We adopt the notions of essential smoothness, essential strict convexity and Legendre type functions from \cite[Section 26]{Roc70}: We say that a function $f \in \Gamma_0(\bR^n)$ is \emph{essentially smooth}, if $\intr(\dom f) \neq \emptyset$ and $f$ is differentiable on $\intr(\dom f)$ such that $\|\nabla f(x^\nu)\|\to \infty$, whenever $\intr(\dom f) \ni x^\nu \to x \in \bdry\dom f$, and \emph{essentially strictly convex}, if $f$ is strictly convex on every convex subset of $\dom \partial f$, and \emph{Legendre type}, if $f$ is both essentially smooth and essentially strictly convex.
Let $F: \bR^n \rightrightarrows \bR^n$ be a set-valued mapping. We define its domain $\dom F := \{x \in \bR^n: F(x) \neq \emptyset\}$
and its range $\ran F := \{u \in \bR^n: u \in F(x) \text{ for some } x \in \bR^n\}$. $F$ is locally bounded at a point $\bar x \in \bR^n$ if for some neighborhood $V$ of $\bar x$ the set $F(V) \subseteq \bR^n$ is bounded. It is called locally bounded on $\bR^n$ if this holds for every $\bar x \in \bR^n$. For a function $f$ defined on a nonempty, convex and open set $C \subseteq \bR^n$ that is proper, lsc and convex and $\mathcal{C}^1$ on $\intr \dom f = C$, we define $D_f(x,y) = f(x) - f(y) - \langle \nabla f(y),x-y \rangle$ if $x \in \dom f$, $y \in \intr \dom f$ and $+\infty$ otherwise, called the Bregman divergence generated by $f$. We denote by $\bN_0 := \bN \cup \{0\}$.
Otherwise we adopt the notation from \cite{RoWe98}. 

\section{\texorpdfstring{$\Phi$}{Φ}-convexity}
\label{sec:gen_convexity}
In this section we recapitulate the existing notions of $\Phi$-convexity and $\Phi$-conjugacy \cite{moreau1970inf} which are used heavily as tools in the remainder of the manuscript. Originating as a generalization of convexity to nonlinear spaces, these notions have since appeared in the context of eliminating duality gaps in nonconvex and nonsmooth optimization \cite{rockafellar1974augmented,penot1990strongly,Vil08,bauermeister2021lifting} and optimal transport theory \cite{Vil08}.

\begin{definition}[$\Phi$-convex and $\Phi$-concave functions] \label{def:phi_cvx}
Let $X$ and $Y$ be nonempty sets and $\Phi: X \times Y \to \bR$ a real-valued coupling. Let $f : X \to \exR$ and $g: Y \to \exR$.
We say that $f$ is $\Phi$-convex on $X$ if there is an index set $\cI$ and parameters $(y_i, \beta_i) \in Y \times \exR$ for $i \in \cI$ such that
\begin{align}
f(x) = \sup_{i \in \cI} \Phi(x, y_i) - \beta_i\quad \forall x \in X.
\end{align}
When $\cI =\emptyset$ we define $f \equiv -\infty$. Likewise we say that $g$ is $\Phi$-convex on $Y$ if there is an index set $\cJ$ and parameters $(x_j, \alpha_j) \in X \times \exR$ for $j \in \cJ$ such that
\begin{align}
g(y) = \sup_{j \in \cJ} \Phi(x_j, y) - \alpha_j \quad \forall y \in Y.
\end{align}
When $\cI =\emptyset$ we define $g \equiv -\infty$. We say that $f$ or $g$ is $\Phi$-concave if $-f$ or $-g$ is $\change{\Phi}{(-\Phi)}$-convex.
\end{definition}
\begin{figure}[!h] 
    \centering
    \includegraphics[]{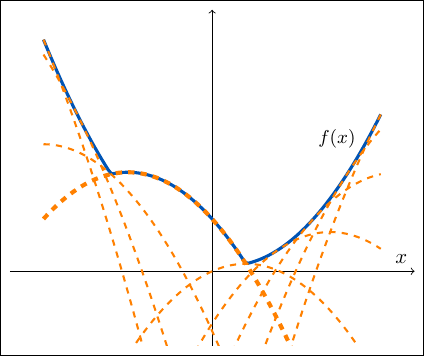}
    \caption{Illustration of the nonconvex function $f(x) = \max\{x^2 - x + 1, -x^2-5x+7\}$} which is $\Phi$-convex for $\Phi(x,y) = -(x-y)^2$, along with its $\Phi$-minorants at various points.\label{fig:phi_conv}
\end{figure}
If $X=Y$ are indistinguishable and $\Phi$ is not symmetric we shall refer to $f$ as left $\Phi$-convex and $g$ as right $\Phi$-convex. Note that in the bibliography $\Phi$-convex functions appear also as $\Phi$-envelopes. In light of \cite[Theorem 8.13]{RoWe98}, any proper, lsc and convex function $h:\bR^n \rightarrow \exR$ is the pointwise supremum over a family of affine functions:
\begin{equation}
    h(x) = \sup_{i \in \mathcal{I}}\langle x, y_i\rangle - \beta_i.
\end{equation}
Therefore, using the coupling $\Phi = \langle \cdot,\cdot \rangle$ and identifying $X, Y$ with $\bR^n$, one recovers from $\Phi$-convex functions the class of proper, lsc and convex functions and $(-\Phi)$-concavity coincides with the classical notion of concavity. An example of a nonconvex function that is yet $\Phi$-convex is provided in \cref{fig:phi_conv}. Note that the minorants of the function in this case are nonlinear functions, in contrast to classical convexity.

We also present the notion of a $\Phi$-conjugate and a $\Phi$-biconjugate function:
\begin{definition}[$\Phi$-conjugate functions]
Let $X$ and $Y$ be nonempty sets and $\Phi: X \times Y \to \bR$ a real-valued coupling. Let $f: X \to \exR$. Then we define
\begin{align}
f^\Phi(y)=\sup_{x \in X} \Phi(x, y) - f(x),
\end{align}
as the $\Phi$-conjugate of $f$ on $Y$ and 
\begin{align}
f^{\Phi\Phi}(x)=\sup_{y \in Y} \Phi(x, y) - f^\Phi(y),
\end{align}
as the $\Phi$-biconjugate back on $X$.
The definitions of $g^\Phi$ and $g^{\Phi\Phi}$ for $g:Y \to \exR$ are parallel.
\end{definition}
Equivalently to the $\Phi$-convexity definition, if $X=Y$ are indistinguishable and $\Phi$ is not symmetric we shall refer to $f^\Phi$ as the left $\Phi$-conjugate and $g^\Phi$ as the right $\Phi$-conjugate.
From the definition it is clear that $f^\Phi$ is $\Phi$-convex on $Y$ and $f^{\Phi\Phi}$ is $\Phi$-convex back on $X$. Notice that when $X$ is a Banach space and $X^*$ its dual, by considering the duality pairing $\langle \cdot, \cdot \rangle_{X^* \times X}$ as a coupling, we obtain the classical convex conjugate definition. An interesting example of $\Phi$-conjugacy is explored in \cite[Example 11.66]{RoWe98}. Let $\Phi(x,(y,r)) = \langle y,x \rangle - \tfrac{r}{2}\|x\|^2$ and identify $X$ with $\bR^n$ and $Y$ with $\bR^n \times \bR$. In this case, we have $f^\Phi(y,r) = (f+\tfrac{r}{2}\|\cdot\|^2)^*(y)$ and if $f$ is prox-bounded, $f^\Phi$ is proper, lsc and convex on $Y$.

We also adopt from \cite{balder1977extension} the definition of the $\varepsilon$-$\Phi$-subgradients, which is a generalization of the classical notion of $\varepsilon$-subgradients:

\begin{definition}[$\varepsilon$-$\Phi$-subgradients] \label{def:phi_subg}
Let $X$ and $Y$ be nonempty sets and $\Phi: X \times Y \to \bR$ a real-valued coupling. Let $f: X \to \exR$ and $\varepsilon \geq 0$. Then we say that $y$ is a $\varepsilon$-$\Phi$-subgradient of $f$ at $\bar x$ if
\begin{equation} \label{eq:phi_subgradient_ineq}
    f(x) \geq f(\bar x) + \Phi(x, y) - \Phi(\bar x, y) - \varepsilon,
\end{equation}
for all $x \in X$.

We denote by $\partial_\Phi^\varepsilon f(\bar x)$ the set of all $\varepsilon$-$\Phi$-subgradients of $f$ at a point $\bar{x} \in X$.

If $\varepsilon =0$ we omit the superscript $\varepsilon$ and refer to $\partial_\Phi f(\bar x)$ as the $\Phi$-subdifferential of $f$ at $\bar x$ and $y \in \partial_\Phi f(\bar x)$ as a $\Phi$-subgradient of $f$ at $\bar x$.

If $\partial_\Phi f(\bar x) = \{y\}$ is a singleton we adopt the notation $\nabla_\Phi f(\bar x) = y$.
\end{definition}
If $\Phi(x,y)=\langle x,y \rangle$, the above defintion coincides with the classical notion of $\varepsilon$-subgradients of convex functions.

The following statement is a generalization of the Fenchel--Young theorem. It is standard in literature; see \cite{balder1977extension,dolecki1978convexity} and references therein. 
\begin{proposition} \label{thm:phi_envelope}
Let $X$ and $Y$ be nonempty sets, $\Phi: X \times Y \to \bR$ a real-valued coupling and $f:X\to \exR$. Then we have
\begin{propenum}
\item \label{thm:phi_envelope:phi_convex} $f^\Phi$ is $\Phi$-convex on $Y$ and $f^{\Phi\Phi}$ is $\Phi$-convex on X;
\item \label{thm:phi_envelope:fenchel_young} $f(x) + f^\Phi(y) \geq \Phi(x, y) \quad \forall x \in X, \quad \forall y \in Y$;
\item \label{thm:phi_envelope:phi_biconj} $f(x) \geq f^{\Phi\Phi}(x) \quad \forall x \in X$.
\end{propenum}
In addition, $f^{\Phi\Phi}$ is the pointwise largest $\Phi$-convex function below $f$. In particular, this means that $f$ is $\Phi$-convex on $X$ if and only if $f(x) = f^{\Phi\Phi}(x)$ for all $x \in X$.

The statements for $g:Y \to \exR$ are parallel.
\end{proposition}

The following proposition relates $\Phi$-subgradients with minimization problems and is of major importance for the analysis of the proposed framework. It can be found in \cite[Lemma 2.6]{laude2023dualities}.
\begin{proposition} \label{thm:phi_subgradients}
    Let $X$ and $Y$ be nonempty sets, $\Phi: X \times Y \to \bR$ a real-valued coupling and  $f: X \to \exR$. Then for any $\bar x \in X$ and $\bar y \in Y$ the following statements are equivalent:
    \begin{propenum}
\item $\bar y \in \partial_\Phi f(\bar x)$; \label{thm:phi_subgradients:y_subg}
\item $ f(\bar x) + f^\Phi(\bar y) = \Phi(\bar x, \bar y)$;
\item $\bar x \in \argmin_{x \in X} f(x) - \Phi(x, \bar y)$; \label{thm:phi_subgradients:x_argmin}
\end{propenum}
where any of the above equivalent statements implies that $f^{\Phi\Phi}(\bar x)= f(\bar x)$ and $\bar x \in \partial_\Phi f^\Phi(\bar y)$.
In particular, if $f$ is $\Phi$-convex $\bar y \in \partial_\Phi f(\bar x) \Leftrightarrow \bar x \in \partial_\Phi f^\Phi(\bar y)$.
\end{proposition}

An important notion that will be applied to $\Phi$-subdifferentials is that of a \emph{selection} of a set valued mapping. We next provide its definition from \cite[p. 49]{dontchev2009implicit}:
\begin{definition}[selection]
    Given a set-valued mapping $F: \bR^n \to \bR^m$ and a set $D \subseteq \dom F$, a function $w: \bR^n \to \bR^m$ is said to be a selection of $F$ on $D$ if $\dom w \supseteq D$ and $w(x) \in F(x)$ for all $x \in D$. If moreover $w$ is continuous, it is called a continuous selection.
\end{definition}
\begin{example}[quadratic coupling] \label{example:quadratic}
    A coupling function that is standard in the related literature is $\Phi(x,y) = -\tfrac{L}{2}\|x-y\|^2$ for $L > 0$. For this coupling, in light of \cite[Proposition 3.4]{laude2023dualities}, any $f \in \Gamma_L(\bR^n)$ is a $\Phi$-convex function, while its $\Phi$-subdifferential is given by $\partial_\Phi f(x) = x + \tfrac{1}{L}\partial f(x)$ for any $x \in \dom f$. Note that if $w: \bR^n \to \bR^n$ is a continuous selection of $\partial f$ on $\bR^n$, then $x \mapsto x + \tfrac{1}{L}w(x)$ is a continuous selection of $\partial_\Phi f$ on $\bR^n$. The $\Phi$-conjugate of a function $g:\bR^n \to \exR$ is given in this setting by $g^\Phi(y) = \max_{x \in \bR^n}\Phi(x,y) - g(x) = -\min_{x \in \bR^n}g(x) + \tfrac{L}{2}\|x-y\|^2$, which is the negative Moreau envelope of the function $g$ \cite[Definition 1.22]{RoWe98}. An example of a nonconvex function that is $\Phi$-convex for $L=1$ is given in \cref{fig:phi_conv}.
\end{example}

\section{Difference of \texorpdfstring{$\Phi$}{Φ}-convex problems}
\label{sec:prob_def}
\subsection{Problem definition}
In this section we define the problem that we study in this paper. It can be considered as a generalization of the classical \emph{difference-of-convex} (DC) problem where convexity is replaced by the more general notion of $\Phi$-convexity that was presented in the previous section.
The problem formulation that we consider is the following
\begin{equation} \label{eq:opt} \tag{P}
    \min_{x \in X}~F(x):=g(x) - f(x),
\end{equation}
where $\Phi:X \times Y \to \bR$ is continuous and the functions $f:X \to \exR$ and $g:X \to \exR$ are possibly extended-valued. Overall we have the following requirements
\begin{assumenum}
    \item \label{assum:xy} $X \subseteq \bR^n$ and $Y \subseteq \bR^m$ are nonempty.
    \item \label{assum:phi} $\Phi$ is jointly continuous relative to $X\times Y$.
    \item \label{assum:f} $f:X \to \exR$ is proper and $\Phi$-convex on $X$.
    \item \label{assum:g} $g: X \to \exR$ is proper and lower semi-continuous relative to $X$ and there exists $y\in Y$ and $\beta \in \bR$ such that $g \geq \Phi(\cdot, y) -\beta$.
    \item \label{assum:lower_bound} $\alpha = \inf_{x \in X}F < \infty$.
\end{assumenum}
While we assume that $f$ is $\Phi$-convex we do not require $g$ to be $\Phi$-convex in which case \cref{eq:opt} does not amount to a true difference of $\Phi$-convex problem.
Nevertheless, \cref{eq:opt} is equivalent (in terms of infima) to a true difference of $\Phi$-convex problem obtained by a certain double-min duality explored in the next subsection.

\subsection{Difference of \texorpdfstring{$\Phi$}{Φ}-convex duality}
\label{subsec:duality}
A powerful and frequently utilized concept in DC analysis is the double-min duality, commonly referred to as DC-duality \cite{toland1978duality,mordukhovich2022convex}. This principle asserts that the infimum of the original optimization problem is equivalent to the infimum of its corresponding DC-dual problem, i.e. in the case where $f, g$ are proper, convex and lsc we have that $\inf_{x \in X}g(x)-f(x) = \inf_{y \in Y}f^*(y)-g^*(y)$. This type of duality extends beyond traditional settings and is also applicable in the more abstract framework considered in this paper: First, note that by $\Phi$-convexity of $f$ on $X$ invoking \cref{thm:phi_envelope:phi_convex} we have the identity 
\begin{equation}
    -f(x) = -f^{\Phi\Phi}(x)=\inf_{y \in Y} -\Phi(x,y) + f^\Phi(y),
\end{equation}
and thus we can equivalently rewrite the problem as
\begin{align} \nonumber
\inf_{x \in X}~g(x) - f(x)&= \inf_{x \in X}~g(x)-f^{\Phi\Phi}(x) 
\\
&=\inf_{y \in Y} \inf_{x \in X} ~g(x) -\Phi(x,y) + f^\Phi(y)  \nonumber
\\
&=\inf_{y \in Y} ~f^\Phi(y) - g^\Phi(y) 
\\
&=\inf_{y \in Y}~G(y)\tag{D} \label{eq:dual},
\end{align}
where the second equality follows by interchanging the order of minimization \cite[Proposition 1.35]{RoWe98}. Note that we adopt extended arithmetics if necessary, i.e., $g(x) - f(x) = +\infty$ if both $g(x) = \infty$ and $f(x)=\infty$.
By virtue of \cref{thm:phi_envelope:phi_convex} $f^\Phi$ and $g^\Phi$ are both $\Phi$-convex on $Y$ and hence \cref{eq:dual} amounts to a true difference of $\Phi$-convex problem whose infimum coincides with the one of \cref{eq:opt}.
In particular note that \cref{assum:phi,assum:g} imply the properness and lower semi-continuity of $g^\Phi$ and thus the DC-dual problem complies with the same assumptions. 

The dual function $G$ is in general a lower bound on $F$ along the $\Phi$-subdifferential of $f$. This is captured in the following proposition:
\begin{proposition} \label{thm:lower_bound_dual}
    Let $x \in X$ be such that $\partial_\Phi f(x)$ is nonempty and consider $y \in \partial_\Phi f(x)$. Then, 
    \begin{equation}
        G(y) \leq F(x)
    \end{equation}
\end{proposition}
\begin{proof}
    \begin{align*}
        G(y)
        &= 
        - f(x) + \Phi(x, y) + \inf_{z \in X} ~g(z) - \Phi(z, y)  \\
        & \leq g(x) - f(x),
    \end{align*}
    where the equality follows by the equivalence in \cref{thm:phi_subgradients} since $y \in \partial_\Phi f(x)$ and the inequality follows by choosing $z=x$ in the infimum.
\end{proof}

Although both problems are equivalent in terms of infima, due to nonconvexity, we cannot expect our algorithm to find a global minimum. In the next subsection we will therefore explore necessary and sufficient conditions of local optimality. 

\subsection{Optimality conditions for difference of \texorpdfstring{$\Phi$}{Φ}-convex problems}

Even in the classical difference-of-convex case, 
verifying conditions for global minimality can be challenging due to the involvement of the $\varepsilon$-subdifferentials of the functions $f$ and $g$ \cite[Theorem 4.4]{hiriart1989convex}. The situation remains the same for the more general setting described in \eqref{eq:opt}, which has been extensively studied for special cases of the $\Phi$-coupling function; see for example \cite{fajardo2022subdifferentials}. For the sake of completeness, we present some results that provide sufficient and necessary conditions for global minimality in this broader context, which are a generalization of \cite[Theorem 7.114]{mordukhovich2022convex}.
\begin{proposition}[necessary and sufficient condition for global optimality] \label{thm:global_min}
    Let $x^\star \in X$. If $x^\star$ is a global minimizer of \cref{eq:opt} then $\partial_\Phi^\varepsilon f(x^\star) \subseteq \partial_\Phi^\varepsilon g(x^\star)$ for all $\varepsilon > 0$.
    Conversely, $\partial_\Phi^\varepsilon f(x^\star) \subseteq \partial_\Phi^\varepsilon g(x^\star)$ for all $\varepsilon \geq 0$ implies that $x^\star$ is a global minimizer of \cref{eq:opt}.
\end{proposition}
\begin{proof}
    In virtue of \cref{eq:dual} we have that global optimality of $x^\star$ is equivalent to
    \begin{equation*}
        g(x^\star)-f(x^\star) \leq ~f^\Phi(y) - g^\Phi(y), \quad \forall y \in Y
    \end{equation*}
    which reads as
    \begin{equation} \label{eq:incl_second}
        g(x^\star)+ g^\Phi(y) \leq ~ f(x^\star) + f^\Phi(y),\quad \forall y \in Y.
    \end{equation}
    Now, choose an $\varepsilon > 0$ and consider a $y \in \partial_\Phi^\varepsilon f(x^\star) \subseteq Y$, which exists since $f$ is $\Phi$-convex. By definition,
    \begin{equation*}
        f(x) \geq f(x^\star) + \Phi(x, y) - \Phi(x^\star, y) -\varepsilon \qquad \forall x \in X,
    \end{equation*}
    and thus
    \begin{equation*}
        \Phi(x^\star, y)-f(x^\star) \geq \sup_{x \in X} \{\Phi(x,y) - f(x)\} - \varepsilon = f^\Phi(y) - \varepsilon.
    \end{equation*}
    by definition of the $\Phi$-conjugate. Combining the above inequality with \cref{eq:incl_second} we obtain that
    \begin{align*}
        \Phi(x^\star, y) &\geq f(x^\star) + f^\Phi(y) - \varepsilon \\
        &\geq g(x^\star) + g^\Phi(y) - \varepsilon \\
        &=g(x^\star) + \sup_{x \in X} \{\Phi(x,y) - g(x)\} - \varepsilon,
    \end{align*}
    and thus $\Phi(x^\star, y) \geq g(x^\star) + \Phi(x,y) - g(x) - \varepsilon$ for any $x \in X$ implying that $y \in \partial_\Phi^\varepsilon g(x^\star)$.

    We will prove the converse direction by contradiction. Let $\partial_\Phi^\varepsilon f(x^\star) \subseteq \partial_\Phi^\varepsilon g(x^\star)$ for all $\varepsilon \geq 0$. Suppose that $x^\star$ is not a global minimum of $F$. 

This implies the existence of some $\Bar{x} \in X$ and some $\Bar{\varepsilon} > 0$ such that
\begin{equation} \label{eq:contr}
    -\infty < \inf F < g(\Bar{x}) - f(\Bar{x}) + \bar \varepsilon < g(x^\star) - f(x^\star) < \infty.
\end{equation}
Hence $\bar x \in \dom f$ and thus we can find a $\Bar{y} \in \partial_\Phi^{\Bar{\varepsilon}} f(\Bar{x})$. The $\bar \varepsilon$-$\Phi$-subgradient inequality for $\bar y$ reads
\begin{align*}
    \Phi(x, \Bar{y}) - \Phi(\Bar{x}, \Bar{y}) \leq f(x) - f(\Bar{x}) + \Bar{\varepsilon}, \quad \forall x \in X.
\end{align*}
Adding and subtracting $\Phi(x^\star, \Bar{y})$ and $f(x^\star)$, we obtain
\begin{align*}
    \Phi(x, \Bar{y}) - \Phi(x^\star, \Bar{y}) \leq f(x) - f(x^\star) + \Phi(\Bar{x}, \Bar{y}) - \Phi(x^\star, \Bar{y}) + f(x^\star) - f(\Bar{x})  + \Bar{\varepsilon}, \quad \forall x \in X
\end{align*}
and by defining $\varepsilon := \Phi(\Bar{x}, \Bar{y}) - \Phi(x^\star, \Bar{y}) + f(x^\star) - f(\Bar{x})  + \Bar{\varepsilon}$ we have that $\varepsilon \geq 0$ by the $\bar \varepsilon$-$\Phi$-subgradient inequality at $\bar x$ for $\bar y$. This in turn means that $\Bar{y} \in \partial_\Phi^{\varepsilon} f(x^\star)$, while we also have that
\begin{align*}
    \Phi(\bar x, \Bar{y}) - \Phi(x^\star, \Bar{y}) = -f(x^\star) + f(\Bar{x})  - \Bar{\varepsilon} + \varepsilon > g(\Bar{x})-g(x^\star)  + \varepsilon,
\end{align*}
where the last inequality follows by \eqref{eq:contr}. This means that $g(\Bar{x})< g(x^\star) +\Phi(\bar x, \Bar{y}) - \Phi(x^\star, \Bar{y}) -\varepsilon$ and hence the $\varepsilon$-$\Phi$-subgradient inequality is violated at $\bar x$ for $\varepsilon$ and thus $\Bar{y} \notin \partial_\Phi^{\varepsilon} g(x^\star)$, a contradiction.
\end{proof}

The verification of such conditions is in general intractable and thus in this paper we mainly focus on conditions for local minimality and stationarity, which is common practice in nonconvex optimization. We proceed with a necessary condition for local minimality under the existence of local $\Phi$-subgradients for $f$:
\begin{proposition}[necessary condition for local optimality] \label{thm:local_min}
    Let $x^\star$ be a (local) minimizer of \cref{eq:opt} and let $y^\star \in \partial_\Phi f(x^\star)$. Then $y^\star$ satisfies the $\Phi$-subgradient inequality for $g$ (locally).
\end{proposition}
\begin{proof}
    Since $x^\star$ is a (local) minimium of \cref{eq:opt} we have that $g(x) - f(x) \geq g(x^\star) - f(x^\star)$ for all $x \in X$ (near $x^\star$). This inequality can be rewritten as:
    \begin{equation*}
        g(x) - g(x^\star)  \geq f(x) - f(x^\star).
    \end{equation*}
    Since $y^\star \in \partial_\Phi f(x^\star)$ we have that
    \begin{equation}
     f(x)-f(x^\star) \geq \Phi(x, y^\star) - \Phi(x^\star,y^\star).
    \end{equation}
    for all $x \in X$. Combining this inequality with the above, we obtain that:
    \begin{align} \label{eq:subgradient_ineq_local}
        g(x) - g(x^\star) \geq \Phi(x, y^\star) - \Phi(x^\star,y^\star),
    \end{align}
    for all $x \in X$ (with $x$ near $x^\star$).
\end{proof}
\begin{remark}[local vs. global subgradient inequality]
In contrast to the convex case, the validity of the $\Phi$-subgradient inequality in a neighborhood of $x^\star$ does not imply that the inequality holds on the whole set $X$. This is illustrated in \cite[Example 5.5]{laude2023dualities}. Nevertheless, in many important examples, such an implication holds. This is evident by noting that \eqref{eq:subgradient_ineq_local} implies that $x^\star$ is a local minimizer of $g(\cdot) - \Phi(\cdot, y^\star)$. Therefore, if this function is convex the local minimizer is a global one and as such \eqref{eq:subgradient_ineq_local} holds for all $x \in X$, i.e., $y^\star \in \partial_\Phi g(x^\star)$.
\end{remark}

The implication described in \cref{thm:local_min} is a necessary condition for $x^\star$ to be a local minimizer of $F$ under the existence of a (local) $\Phi$-subgradient of $f$. Nevertheless it is impossible to verify such a condition in practice, since it involves an inequality in a generally unknown neighborhood of $x^\star$. A more practical notion is $\Phi$-criticality, which is given below. As we shall see the algorithm proposed in this work converges to $\Phi$-critical points.
\begin{definition}[$\Phi$-criticality]
    We say that a point $x^\star \in \dom F$ is $\Phi$-critical if $y^\star \in \partial_\Phi f(x^\star) \cap \partial_\Phi g(x^\star) \neq \emptyset$. 
\end{definition}
In the case where $f$ and $\Phi$ are smooth $\Phi$-criticality of $x^\star \in \intr X$ implies standard stationarity and hence $\Phi$-criticality is a necessary condition for local optimality. This is explored in the following proposition.
\begin{proposition}
    Let $x^\star \in \intr X$ be a $\Phi$-critical point and $f$ and $\Phi(\cdot, y^\star)$ be strictly differentiable at $x^\star$. Then, $x^\star$ is a stationary point of $F$, i.e. $0 \in \partial F(x^\star)$.
\end{proposition}
\begin{proof}
    Since $y^\star \in \partial_\Phi f(x^\star)$ by \cref{thm:phi_subgradients}, $x^\star \in \argmin_{x \in X} f(x) - \Phi(x, y^\star)$. This in turn implies through Fermat's rule that $\nabla_x \Phi(x^\star, y^\star) = \nabla f(x^\star)$. Since moreover $y^\star \in \partial_\Phi g(x^\star)$, in light of \cite[Theorem 10.1]{RoWe98} we also have that $0 \in \partial (g(x^\star) - \Phi(x^\star, y^\star))$. From the strict differentiability of $\Phi(x^\star, y^\star)$ and since $g$ is proper and lsc we have from \cite[Exercise 10.10]{RoWe98} that $\partial (g(x^\star) - \Phi(x^\star, y^\star)) = \partial g(x^\star) - \nabla_x \Phi(x^\star, y^\star) $ or that $\nabla_x \Phi(x^\star, y^\star) \in \partial g(x^\star)$ which in turn implies that $0 \in \partial g(x^\star) - \nabla f(x^\star)$. Since $f$ is strictly differentiable at $x^\star$, $0 \in \partial F(x^\star)$ from \cite[Exercise 10.10]{RoWe98}.
\end{proof}
The fact that $x^\star$ has to be in the interior of $X$ might seem restrictive, but it is standard in the general nonconvex setting considered in this paper where $0 \in \partial F(x^\star)$ is not a necessary condition for optimality when $x^\star$ is in the boundary of $X$.

 \section{The difference of \texorpdfstring{$\Phi$}{Φ}-convex algorithm}
 \label{sec:alg}
Akin to the classical difference-of-convex approach we propose the \emph{difference-of-$\Phi$-convex algorithm} ($\Phi$-DCA) for solving \eqref{eq:opt}:
\begin{align} \label{eq:DCA}
\begin{cases}
    y^k \in \partial_\Phi f(x^k) \\
    x^{k+1} \in (\partial_\Phi g)^{-1}(y^k),
\end{cases}
\end{align}
where we refer to the $y$-update as the forward-step and the $x$-update as the backward-step.

As for well-definedness of the scheme we have to ensure that range and domain of forward-step and backward-step are compatible with each other. For that purpose assume that $x^0 \in \dom \partial_\Phi f$ and
\begin{assumenum}[resume]
    \item \label{assum:range} $\emptyset \neq \ran (\partial_\Phi g)^{-1} \subseteq \dom \partial_\Phi f$ and $\emptyset \neq \ran \partial_\Phi f \subseteq \dom (\partial_\Phi g)^{-1}$.
\end{assumenum}
In light of \cref{thm:phi_subgradients} the backward-step can be further rewritten as 
\begin{align}
    (\partial_\Phi g)^{-1}(y^k) = \argmin_{x \in X} g(x) - \Phi(x, y^k)\subseteq \partial_\Phi g^\Phi(y^k),
\end{align}
where the inclusion holds with equality if $g$ is $\Phi$-convex. In particular, $\emptyset \neq \dom (\partial_\Phi g)^{-1}$ implies the second part of \cref{assum:g}.
Thanks to the characterization of the backward-step as a minimization problem, $\ran \partial_\Phi f \subseteq \dom (\partial_\Phi g)^{-1}$ can be ensured if $g - \Phi(\cdot, y)$ is level-bounded for every $y \in \ran \partial_\Phi f$.

Notice further that the $\Phi$-subgradient inequality of $f$ for the forward-step gives rise to a minorizing surrogate for $f$:
\begin{equation}
    f(x) \geq f(x^k) + \Phi(x,y^k) -\Phi(x^k, y^k). \label{eq:y_update_surrogate}
\end{equation}
It is transformed into a majorizing model of the composite cost $F$ by multiplication with $-1$ and adding $g(x)$ to both sides of the inequality:
\begin{equation}
    F(x)=g(x)-f(x) \leq g(x) - f(x^k) -\Phi(x, y^k) + \Phi(x^k, y^k).\label{eq:majorizing_model}
\end{equation}
Hence the combined update in \eqref{eq:DCA} can be understood as a majorize-minimize procedure \eqref{eq:majorizing_model}:
\begin{align}\label{eq:majorize_minimize}
    x^{k+1} \in \argmin_{x \in X} ~g(x) - f(x^k) -\Phi(x, y^k) + \Phi(x^k, y^k).
\end{align}

Thanks to the fact that $g^\Phi$ is $\Phi$-convex on $Y$ we can apply the algorithm to \eqref{eq:dual} with flipped roles of $X$ and $Y$. By $\Phi$-convexity of $f^\Phi$ on $Y$ this yields the following algorithm:
\begin{align} \label{eq:DCA_dual}
\begin{cases}
x^k \in \partial_\Phi g^\Phi(y^k) \\
y^{k+1} \in \partial_\Phi (f^\Phi)^\Phi(x^k).
\end{cases}
\end{align}
In light of \cref{thm:phi_envelope} we have that $(f^\Phi)^\Phi = f$. Since, moreover $(\partial_\Phi g)^{-1} \subseteq \partial_\Phi g^\Phi$ due to \cref{thm:phi_subgradients} it is evident that, up to a cyclic interchange of the updates, \eqref{eq:DCA} is an instance of \eqref{eq:DCA_dual}.
If furthermore, $g$ is $\Phi$-convex on $X$, $(\partial_\Phi g)^{-1} = \partial_\Phi g^\Phi$ and thus both algorithms are equivalent up to an interchange of forward- and backward-step.

\begin{table}[ht]
\begin{tabular}{ |c|c|c|c|c|c| }  
 \hline  \xrowht[()]{15pt}
 $\Phi$-DCA & $X$ & $Y$ & $\Phi(x, y)$ & $\Gamma_\Phi(X)$ & $\partial_\Phi f$ \\ 
  \hline \hline \xrowht[()]{10pt}
 DCA & $\bR^n$ & $\bR^n$ & $\langle x, y \rangle$ 
 & $\Gamma_0(\bR^n)$ & $\partial f$ \\ 
 \hline \xrowht[()]{10pt}
 PGM &$\bR^n$ &$\bR^n$ & $-\tfrac{L}{2}\|x-y\|^2$ 
 & $\supseteq\mathcal{C}_L$ & $\id + \tfrac{1}{L}\nabla f$ \\ 
 \hline \xrowht[()]{10pt}
 B-PGM & $\dom h$& $\intr \dom h$ & $-L D_h(x,y)$ & $\supseteq \mathcal{C}_h^B$ & $\nabla h^* \circ (\nabla h+\tfrac{1}{L}\nabla f)$ \\ 
  \hline \xrowht[()]{10pt}
 NGD & $\bR^n$& $\bR^n$ & $-D_h(y, x)$ & $\supseteq \mathcal{C}_{h^*}^{B}$ & $\id+\nabla^2 h(\cdot)^{-1}\nabla f$ \\ 
 \hline \xrowht[()]{10pt}
 a-PGM &$\bR^n$ &$\bR^n$ & $-\tfrac{1}{L} \star h(x-y)$ & $\supseteq \mathcal{C}_h^a$ & $\id - \tfrac{1}{L}\nabla h^* \circ (-\nabla f)$ \\ 
 \hline \xrowht[()]{10pt}
 H-PGM &$\bR^n$ &$\bR^n \times \bR^n$ & \makecell{$\langle x-y_1,y_2 \rangle$\\
 $-\tfrac{H}{\nu+1}\|x-y_1\|^{\nu+1}$} & $\supseteq \mathcal{C}_H^\nu$ & $\supseteq (\id, \nabla f)$ \\ 
 \hline
 \xrowht[()]{10pt}
 TM &$\bR^n$ &$\bR^n \times \bR^n \times \dots$ & \makecell{$\sum_{i=1}^p \tfrac{1}{i!} Y_{i+1} [x - y_1]^i$ \\ $- \tfrac{L_p}{(p+1)!}\|x-y_1\|^{p+1}$}   & $\supseteq \mathcal{C}_L^p$ &  $ \supseteq (\id, \nabla f, \dots, D^p f)$ \\ 
 \hline
 
\end{tabular}
\caption{ \label{tab:algs}
    Examples of $\Phi$-convexity and the $\Phi$-DCA. Here, $\Gamma_\Phi(X)$ denotes the class of functions that are $\Phi$-convex on $X$. The $\Phi$-DCA column denotes the name of the specific instance of the algorithm, as given in the examples of \cref{sec:alg}. $\mathcal{C}_L$ denotes the class of $L$-Lipschitz smooth functions on $\bR^n$, $\mathcal{C}_h^B$ the class of relatively smooth functions \cite{bolte2018first,lu2018relatively} with Legendre convex reference function $h\in \Gamma_0(\bR^n)$ and constant $L$, $\mathcal{C}_{h^*}^{B}$ the class of functions such that $f\circ \nabla h^* \in \Gamma_0(\bR^n)$, $\mathcal{C}_h^a$ the class of anisotropically smooth functions \cite{laude2021lower,laude2023dualities,laude2022anisotropic} with Legendre convex reference function $h$ and constant $L$, $\mathcal{C}_H^\nu$ the class of H\"older smooth functions of order $\nu$ and modulo $H$ and $\mathcal{C}_L^p$ the class of functions with the $p$-th derivative being Lipschitz continuous. The tensor products in the $\Phi$-coupling simplify in the case $p=2$, where  $\Phi(x,(y_1,Y_2,Y_3)) = \langle Y_2, x-y_1 \rangle + \langle x-y_1, Y_3  (x-y_1)\rangle - \tfrac{L_p}{3!}\|x-y_1\|^3$ for $y_1, Y_2 \in \bR^n$ and $Y_3 \in \bR^{n \times n}$.
    The subset relations in the $\Gamma_\Phi(X)$ column indicate that the problem class is a subset of the class of $\Phi$-convex functions. The $\partial_\Phi f$ column contains (the continuous selection of) the $\Phi$-subdifferential of $f$ for $f$ being restricted to the problem class. The subset relations in the final column demonstrate that the $\Phi$-subdifferential might in fact include more $\Phi$-subgradients than the indicated one.
    }
\end{table}

\subsection{Examples of \texorpdfstring{$\Phi$}{Φ}-convexity and unification of existing algorithms} \label{sec:examples}
In this section we present examples of $\Phi$-coupling functions and the corresponding instances of the $\Phi$-DCA, demonstrating that the proposed framework unifies many important methods from existing literature. A summary is provided in \cref{tab:algs}, which shows how each of the presented algorithms is incorporated in the $\Phi$-DCA framework. More precisely, it contains the sets $X, Y$, the $\Phi$-coupling function, the problem class associated with each method and (the continuous selection of) the $\Phi$-subdifferential of $f$ that corresponds to the forward step of each method. It is important to distinguish the difference between the problem class and the class of $\Phi$-convex functions, denoted by $\Gamma_\Phi(X)$, that is in general a superset of the problem class. In order to demonstrate that, consider the case where $\Phi(x,y) = -\tfrac{L}{2}\|x-y\|^2$. In \cref{tab:algs}, the problem class is that of $L$-Lipschitz smooth functions on $\bR^n$, $\mathcal{C}_L$. Nevertheless, in light of \cite[Proposition 3.4]{laude2023dualities}, every lsc and $L$-weakly convex function is $\Phi$-convex in this setting, which implies that $\mathcal{C}_L \subset \Gamma_\Phi(\bR^n)$. 

Note that in the examples which are of the standard proximal gradient type, there is a switch in the sign in front of $f$ in relation to the methods in the literature. This is due to the problem considered in this paper being $g-f$ instead of the standard $g+f$.

\begin{example}[standard difference of convex method (DCA)] \label{example:DCA}
    In the case where both $f,g: \bR^n \to \exR$ are proper, lsc and convex in \eqref{eq:opt}, the $\Phi$-DCA is exactly the standard DCA. The well-definedness assumptions in the setting of this paper are then the ones of the standard DCA (see \cite{tao1997convex})
\end{example}
\begin{example}[proximal gradient method (PGM)] \label{example:PGM}
    Let $X = Y = \bR^n$ and consider in \eqref{eq:opt} the case where $-f$ is $L$-smooth on $\bR^n$, $g$ is proper, lsc, $g + \tfrac{L}{2}\|\cdot-y\|^2$ is level-bounded for any $y \in \bR^n$ and choose $\Phi(x,y) = -\tfrac{L}{2}\|x-y\|^2$. 
    \begin{itemize}
        \item The standard Euclidean descent inequality can be shown to be the $\Phi$-subgradient inequality for $f$ \eqref{eq:phi_subgradient_ineq} using some simple algebraic manipulations:
    \begin{equation*}
        f(x) \geq f(\bar x) + \tfrac{L}{2}\|\tfrac{1}{L} \nabla f(\bar x)\|^2 - \tfrac{L}{2}\|x-(\bar x + \tfrac{1}{L} \nabla f(\bar x))\|^2 \qquad x, \bar x \in \bR^n.
    \end{equation*}
    From the inequality above it is evident that $\bar y := \bar x + \tfrac{1}{L}\nabla f(\bar x) \in \partial_\Phi f(\bar x)$ for all $\bar x \in \bR^n$. More precisely, from \cite[Proposition 3.4]{laude2023dualities} $\bar y$ is the unique element of $\partial_\Phi f(\bar x)$.
        \item Regarding the backward step, due to the level-boundedness of $g+\tfrac{L}{2}\|\cdot-y\|^2$ for all $y \in \bR^n$, $\argmin_x g(x)+\tfrac{L}{2}\|x-\cdot\|^2$ is well-defined in $\bR^n$. This fact, along with $\dom \partial_\Phi f = \bR^n$ implies that \cref{assum:range} is satisfied and as such the algorithm \eqref{eq:majorize_minimize} is well-defined and given by
        \begin{equation*}
            x^{k+1} \in \argmin_{x \in \bR^n} g(x) + \tfrac{L}{2}\|x - (x^k + \tfrac{1}{L}\nabla f(x^k))\|^2,
        \end{equation*}
        which is the standard proximal gradient step.
    \end{itemize}    
\end{example}
\begin{example}[Bregman proximal gradient method (B-PGM)] \label{example:BPGM}
    Consider now the case where $-f$ is $L$-smooth relative to a Legendre function $h:\bR^n \to \bR$ with $\intr \dom h$ a nonempty and convex set. More precisely identify $X$ with $\dom h$ and assume moreover that $h$ is continuous on $X$, while the following inequality holds:
    \begin{equation*}
        -f(x) \leq -f(\bar x) + \langle - \nabla f(\bar x),x-\bar x \rangle + L D_h(x, \bar x),
    \end{equation*}
    for $x \in X$ and $\bar x \in \intr \dom h$. Examples of such functions $h$ can be found in \cite[Example 3.1]{teboulle1992entropic}, while examples of relative smoothness can be found in \cite{bolte2018first, lu2018relatively}.  In this case the Bregman proximal gradient method is a specific instance of the $\Phi$-DCA for $\Phi(x,y) = -L D_h(x, y)$ and $Y = \intr \dom h$.
    \begin{itemize}
        \item The inequality above after some algebraic manipulations can be transformed into the $\Phi$-subgradient inequality for $f$:
        \begin{equation*}
            f(x) \geq f(\bar x) - LD_h(x, \nabla h^*(\nabla h(\bar x)-\tfrac{1}{L}\nabla f(\bar x))) + L D_h(\bar x, \nabla h^*(\nabla h(\bar x)-\tfrac{1}{L}\nabla f(\bar x))),
        \end{equation*}
        for all $x \in X$ and $\bar x \in Y$,
        which directly implies that $\nabla h^*(\nabla h(\bar x)-\tfrac{1}{L}\nabla f(\bar x)) \in \partial_\Phi f(\bar x)$.
        This fact was studied in \cite[Proposition 3.4]{laude2023dualities} where it is shown that the $\Phi$-subdifferential of $f$ is exactly given by $\partial_\Phi f(x) = \{\nabla h^*(\nabla h(x) + \tfrac{1}{L}\nabla f(x))\}$ for $x \in \intr \dom h$, i.e. the forward step of the Bregman gradient method is the unique element of $\partial_\Phi f(\bar x)$ for all $\bar x \in \intr X$.

        \item Now, notice that by using the formula for $\partial_\Phi f(x^k)$ in \eqref{eq:majorize_minimize} and discarding the constant terms we obtain the following update:
        $$
            x^{k+1} \in \argmin_{x \in X} g(x) - \langle \nabla f(x^k),x-x^k \rangle + L D_h(x, x^k),
        $$
        which is exactly the Bregman proximal gradient update. Regarding the well-definedness of the iterates, it can be verified that the condition in \cref{assum:range} is guaranteed from \cite[Assumptions B and C]{bolte2018first} along with the nonemptyness of the $\Phi$-subdifferential of $f$ on $\intr X$.
    \end{itemize}
\end{example}
\begin{example}[natural gradient method (NGM) \cite{amari1998natural}] \label{example:NGM}
    When considering the Bregman distance $D_h$ as a coupling function in the $\Phi$-DCA, one can switch the role of the arguments $x$ and $y$ and obtain the well-studied natural gradient algorithm. This fact was extensively studied in \cite[Section 4.2]{léger2023gradient} under a general alternating minimization framework. More specifically, in \eqref{eq:opt} assume now that $g := 0$ and $h^*+f\circ\nabla h^*$ defined on $\bR^n$ is a convex function. Then we have $\dom \partial_\Phi f = \bR^n$ for $\Phi(x,y) = -D_h(y, x)$ and the unique $\Phi$-subgradient is given by $\partial_\Phi f(x) = \{x + \nabla^2 h(x)^{-1}\nabla f(x)\}$ as proved in the following proposition.
\end{example}
    \begin{proposition}
        Let $f:\bR^n \to \bR$ be smooth, $h \in \mathcal{C}^3(\bR^n)$ a strictly convex function with nonsingular Hessian and $h^*+f\circ\nabla h^*$ be convex. Then,
        $$\partial_\Phi f(x) = \{x + \nabla^2 h(x)^{-1}\nabla f(x)\}  \qquad \forall x \in \bR^n.$$
    \end{proposition}
    \begin{proof}
        Let us consider the convex gradient inequality for $h^* + f \circ \nabla h^*$ between points $x, \bar x \in \bR^n$:
        \begin{align*}
            f(\nabla h^*(x)) + h^*(x) \geq f(\nabla h^*(\bar x)) + h^*(\bar x) + \langle \nabla^2 h^*(\bar x)\nabla f(\nabla h^*(\bar x)) + \nabla h^*(\bar x), x- \bar x \rangle.
        \end{align*}
        With the change of variables $\nabla h^*(x) = z$ and $\nabla h^*(\bar x) = \bar z$, we get
        \begin{align*}
        f(z) 
        &\geq
        f(\bar z) - h^*(\nabla h(z)) + h^*(\nabla h(\bar z)) + \langle \nabla^2 h(\bar z)^{-1}\nabla f(\bar z) + \bar z, \nabla h(z)- \nabla h(\bar z) \rangle \\
        &=
        f(\bar z) + \langle \nabla^2 h(\bar z)^{-1}\nabla f(\bar z), \nabla h(z)- \nabla h(\bar z) \rangle - h^*(\nabla h(z)) + h^*(\nabla h(\bar z)) + \langle \bar z, \nabla h(z)- \nabla h(\bar z) \rangle \\
        &=
         f(\bar z) + \langle \nabla^2 h(\bar z)^{-1}\nabla f(\bar z), \nabla h(z)- \nabla h(\bar z) \rangle - D_{h^*}(\nabla h(z), \nabla h(\bar z)) \\
        &=
         f(\bar z) + \langle \nabla^2 h(\bar z)^{-1}\nabla f(\bar z), \nabla h(\bar z)- \nabla h(z) \rangle - D_{h}(\bar z, z),
    \end{align*}
    where in the last equality we have used the fact that $D_{h^*}(\nabla h(z), \nabla h(\bar z)) = D_{h}(\bar z, z)$, which follows by the strict convexity of $h$. Now, consider points $w = z$, $\bar w = \bar z$ and $w^+ = \bar z + \nabla^2 h(\bar z)^{-1}\nabla f(\bar z)$ and note that from the three-point property we have:
    \begin{align*}
        \langle \nabla^2 h(\bar z)^{-1}\nabla f(\bar z), \nabla h(\bar z)- \nabla h(z) \rangle - D_{h}(\bar z,z) 
        &=
        -\langle \bar w- w^+,\nabla h(\bar w)- \nabla h(w) \rangle + D_h(\bar w, w) \\
        &= -D_h(w^+, w) + D_h(w^+, \bar w)
    \end{align*}
    Substituting back into the inequality above we get
    \begin{equation}
        f(z) + D_h(\bar z + \nabla^2 h(\bar z)^{-1}\nabla f(\bar z), z) \geq f(\bar z) + D_h(\bar z + \nabla^2 h(\bar z)^{-1}\nabla f(\bar z), \bar z),
    \end{equation}
    which is the $\Phi$-subgradient inequality. This directly implies that $\bar x + \nabla^2 h(\bar x)^{-1}\nabla f(\bar x) \in \partial_\Phi f(\bar x)$. Regarding the uniqueness of the $\Phi$-subgradients, in light of \cref{thm:phi_subgradients}, $\bar y \in \partial_\Phi f(\bar x)$ is equivalent to $\bar x \in \argmin_{x \in X} f(x) + D_h(\bar y, x)$. The optimality conditions for this minimization problem give us $\nabla_y D_h(\bar y, \bar x) = \nabla f(\bar x)$ or that $\bar y$ is uniquely defined by $\bar x$.
    \end{proof}  
    Since $g = 0$ the primal step in \eqref{eq:majorize_minimize} just imposes that
    \begin{equation}
        x^{k+1} = x^k + \nabla^2 h(x^k)^{-1}\nabla f(x^k),
    \end{equation} 
    which is the Natural gradient method with the metric induced by $h$.

    \begin{example}[anisotropic proximal gradient method (a-PGM) \cite{laude2022anisotropic}] \label{example:apgm} In this example, $-f$ in \eqref{eq:opt} is anisotropically smooth \cite[Definition 2.1]{laude2022anisotropic} relative to some Legendre convex function $h$ such that $\dom h = \bR^n$ with constant $L > 0$, while $g$ is proper, lsc and such that $g+\tfrac{1}{L}\star h(\cdot-y)$ is level-bounded for any $y \in \bR^n$. This algorithm can be incorporated into the $\Phi$-DCA framework by choosing $X = Y = \bR^n$ and $\Phi(x,y) = -\tfrac{1}{L}\star h(x-y)$, as shown in \cite[Section 6]{laude2022anisotropic}.
    \begin{itemize}
        \item Regarding the forward step, the anisotropic descent inequality \cite[Definition 2.1]{laude2022anisotropic} can be written as:
    \begin{equation*}
        f(x) \geq f(\bar x) - \tfrac{1}{L}\star h(x-\bar x + \tfrac{1}{L}\nabla h^*(-\nabla f(\bar x))) + \tfrac{1}{L}\star h(\tfrac{1}{L}\nabla h^*(-\nabla f(\bar x)))
    \end{equation*}
    The inequality above along with \eqref{eq:phi_subgradient_ineq} implies that $\bar y= \bar x - \tfrac{1}{L}\nabla h^*(-\nabla f(\bar x)) \in \partial_\Phi f(\bar x)$, while from \cite[Lemma D.2]{laude2022anisotropic} we know that $\partial_\Phi f(\bar x)$ is a singleton.
    
    \item By using the formula for the forward step in \eqref{eq:majorize_minimize}, the main iteration becomes:
    \begin{equation*}
        x^{k+1} \in \argmin_{x \in \bR^n} g(x) + \tfrac{1}{L}\star h(x-x^k + \tfrac{1}{L}\nabla h^*(-\nabla f(x^k))).
    \end{equation*}
    It is once again evident that the level-boundedness of $g+\tfrac{1}{L}h(\cdot-y)$ for any $y \in \bR^n$ guarantees that \cref{assum:range} holds and as such the algorithm is well-defined.
    \end{itemize}
    \end{example}

    \begin{example}[Methods for $(L_0, L_1)$-smoothness]
        The anisotropic proximal gradient method that is presented in \cref{example:apgm} above was recently shown in \cite{oikonomidis2025nonlinearly} to encompass methods for $(L_0, L_1)$-smoothness \cite[Definition 1]{zhanggradient}. We say that $f \in \cC^2(\bR^n)$ is $(L_0, L_1)$-smooth if
        \begin{equation*}
            \|\nabla^2 f(x)\| \leq L_0 + L_1\|\nabla f(x)\|,
        \end{equation*}
        with $L_0, L_1 > 0$. Note that in light of \cite[Corollary 2.7]{oikonomidis2025nonlinearly} and \cite[Proposition 2.9]{oikonomidis2025nonlinearly}, a $(L_0, L_1)$-smooth function $f$ is $(\delta L_1, L_0/L_1)$-anisotropically smooth relative to $h = -\|\cdot\| - \ln(1-\|\cdot\|)$ (see \cite[Definition 2.1]{oikonomidis2025nonlinearly}) for any $\delta < 1$. 
        
        Consider now \eqref{eq:opt} with $g=0$ and $-f$ $(L_0, L_1)$-smooth. Due to the reasoning of the previous paragraph and \cite[Proposition 2.2]{oikonomidis2025nonlinearly}, $-f$ is $\Phi$-convex with coupling $\Phi(x,y) = - L_0/L_1 ((\delta L_1^{-1}) \star h)(x-y)$ for $h$ as above. The $\Phi$-subdifferential is then given by the forward operator of \cite[Corollary 2.7]{oikonomidis2025nonlinearly}:
        \begin{equation*}
            \partial_\Phi f(x) = \{ x + \frac{\delta}{L_0 + L_1\|\nabla f(x)\|}\nabla f(x)\}.
        \end{equation*}
        Since $g=0$, the main iteration of the $\Phi$-DCA becomes
        \begin{equation*}
            x^{k+1} = x^k + \frac{\delta}{L_0 + L_1\|\nabla f(x^k)\|}\nabla f(x^k),
        \end{equation*}
        thus recovering the algorithm proposed in \cite[Algorithm 1]{gorbunov2024methods} for convex $(L_0, L_1)$-smooth functions.
    \end{example}

    \begin{example}[H\"older proximal gradient method (H-PGM)] \label{example:holder}
        In the previous examples $X$ and $Y$ were subsets of the same Euclidean space. Nevertheless, the $\Phi$-convexity framework allows for more robust parametrizations. Let us consider the case where now $-f$ has H\"older continuous gradients, $\|\nabla f(x) - \nabla f(\bar x)\| \leq H \|x-\bar x\|^\nu$ for $H > 0$, $\nu \in (0, 1]$ and $g$ is such that $g + \tfrac{H}{\nu+1}\|\cdot - y\|^{\nu + 1}$ is level-bounded for any $y \in \bR^n$. 
        \begin{itemize}
            \item The H\"older continuity inequality provides the following upper bound for $-f$ \cite[Equation (15)]{devolder2014first}:
            \begin{equation*}
                -f(x) \leq -f(\bar x) - \langle \nabla f(\bar x),x-\bar x \rangle + \tfrac{H}{\nu+1}\|x - \bar x\|^{\nu + 1}.
            \end{equation*}
            From the inequality above it is straightforward that for $y = (y_1, y_2) \in \bR^n \times \bR^n$, $X = \bR^n$ and $\Phi(x, y) = \langle x - y_1, y_2 \rangle - \tfrac{H}{\nu+1}\|x-y_1\|^{\nu+1}$ we have that $(\bar x, \nabla f(\bar x)) \in \partial_\Phi f(\bar x)$ for all $x \in X = \bR^n$. This specific choice of $\Phi$-subgradients shows the existence of a selection $\phif: \bR^n \to \bR^n \times \bR^n$ of $\partial_\Phi f$ on 
            $X$, given by $\phif (x) = (x, \nabla f(x))$. Since $\nabla f$ is continuous, $\phif$ is a continuous selection. Note that when $g = 0$, choosing $y^k = \phif(x^k)$ in \eqref{eq:DCA}, gives the standard H\"older gradient descent method \cite{bolte2020ah}.

            \item Choosing $y^k = \phif (x^k)$ and substituting back into \eqref{eq:DCA}, we obtain the following update:
            $$
                x^{k+1} \in \argmin_{x \in \bR^n} g(x) - \langle \nabla f(x^k),x-x^k \rangle + \tfrac{H}{\nu+1}\|x-x^k\|^{\nu+1},
            $$
            which in the convex case is the algorithm analyzed in \cite{bredies2008forward} under a general Banach space setting.
        \end{itemize}
    \end{example}

    \begin{example}[exact tensor methods (TM)] \label{example:tensor}
        In the following we adopt the notation from \cite{doikov2020inexact}. Note that $D^p f$ denotes the $p$-th derivative of $f$ and $D^p f(x)[h_1, \dots, h_p]$ the $p$-th directional derivative of $f$ along directions $h_1, \dots, h_p \in \bR^n$. WHen $h_i = h$ for all $1 \leq i \leq p$, the notation $D^p f(x)[h]^p$ is used. Consider the case studied in that paper where $g$ is proper, lsc and convex and the $p$-th derivative of $-f$ is Lipschitz continuous, i.e. $\|D^p f(x) - D^pf(\bar x)\| \leq L_p \|x-\bar x\|$ for all $x, \bar x \in \bR^n$. This implies the following inequality, which describes an upper bound for $-f$ \cite[Equation (3)]{doikov2020inexact}:
        \begin{equation*}
            -f(x) + f(\bar x) + \sum_{i=1}^p \tfrac{1}{i!}D^i f(\bar x) [x - \bar x]^i \leq \tfrac{L_p}{(p+1)!}\|\bar x - x\|^{p+1} \qquad \forall x, \bar x \in \bR^n.
        \end{equation*}
        Therefore, $f$ has an everywhere nonempty $\Phi$-subdifferential for 
        $$\Phi(x,(y_1, Y_2, \ldots, Y_p)) =  \sum_{i=1}^p \tfrac{1}{i!} Y_{i+1} [x - y_1]^i - \tfrac{L_p}{(p+1)!}\|x-y_1\|^{p+1}
        $$
        and $(\bar x, \nabla f(\bar x), \dots, D^p f(\bar x)) \in \partial_\Phi f(\bar x)$. By choosing this specific $\Phi$-subgradient in the $\Phi$-DCA we get the following update:
        \begin{equation}
            x^{k+1} \in \argmin_{x \in \bR^n} g(x) - \sum_{i=1}^p \tfrac{1}{i!}D^i f(x^k) [x - x^k]^i + \tfrac{L_p}{(p+1)!}\|x - x^k\|^{p+1},
        \end{equation}
        which is exactly the update in \cite[Equation (5)]{doikov2020inexact}. As in \cref{example:holder}, the choice $\phif(x) = (x, \nabla f(x), \dots, D^p f(x)) \in \partial_\Phi f(x)$ amounts to a continuous selection of $\partial_\Phi f$ for all $x \in \bR^n$.
    \end{example}

\section{Convergence analysis}
\label{sec:conv_analysis}
\subsection{Asymptotic convergence analysis}
Having demonstrated the generalization properties of the proposed framework, we now move on to its convergence analysis. In order to quantify the progress of the algorithm per iteration, we define the following regularized \textit{gap function} $\gapf : X \times Y \to \exR$, which we use as a measure of stationarity:
\begin{equation} \label{eq:gap}
    \gapf(x, y) := g(x) + g^\Phi(y) -\Phi(x, y).
\end{equation}
for any $(x, y) \in X \times Y$. This function is a generalization of the gap in \cite[Equation (13)]{karimi2016linear}, which was further utilized in \cite{laude2022anisotropic} in order to prove the asymptotic convergence of the proposed method. It has also been considered as a generalized Fenchel-Young loss in \cite{blondel2022learning}. In this paper we go beyond the analysis in the aforementioned works in that we utilize both the quantity defined in \eqref{eq:gap} and the one corresponding to the dual problem \eqref{eq:dual}, thus obtaining tighter results.
Note that the definition of $\gapf$ can be given in terms of the value function of the $\Phi$-DCA, 
\begin{equation} \label{eq:value_func}
   \valf(x, y) := \inf_{z \in X} ~g(z) - f(x) - \Phi(z, y) + \Phi(x, y),
\end{equation}
via the relation $\gapf(x, y) = F(x) - \valf(x, y)$:
\begin{align*}
    \gapf(x, y) &= g(x) + \sup_{z \in X}\{\Phi(z,y) - g(z)\} -\Phi(x, y)
    \\
    &= F(x) + f(x) -\Phi(x, y) - \inf_{z \in X}\{g(z) - \Phi(z,y)\} = F(x) - \valf(x, y)
\end{align*}
The value function generalizes the notion of the forward-backward envelope \cite{stella2017forward} to the broader setting studied in this paper, a fact which can be seen by specializing its definition to the setting described in \cref{example:PGM}.

Next we show some key properties of the gap function that justify its usage as a measure of stationarity:
\begin{proposition}\label{thm:gaps}
The gap function $\gapf(x, y) \geq 0$ for all $x \in X$ and $y \in Y$ is lsc. For any $(x^\star, y^\star) \in \gph \partial_\Phi f$, $\gapf(x^\star, y^\star) = 0$ if and only if $x^\star$ is a $\Phi$-critical point.
\end{proposition}
\begin{proof}
    By the definition of $\gapf(x, y)$ and \cref{thm:phi_envelope}, we have that $\gapf(x, y) \geq 0$ for all $(x, y) \in \dom \gapf$. Moreover, $g$ is lsc by assumption, $g^\Phi$ is lsc as the pointwise supremum over continuous functions and $\Phi$ is continuous, implying that $\gapf$ is lsc. 
    Notice that $\gapf(x^\star, y^\star) = 0$ implies that $g^\Phi(y^\star) = \Phi(x^\star, y^\star) - g(x^\star)$, which by \cref{thm:phi_subgradients} further means that $y^\star \in \partial_\Phi g(x^\star)$. Since $(x^\star, y^\star) \in \gph \partial_\Phi f$ we have that $y^\star \in \partial_\Phi f(x^\star) \cap \partial_\Phi g(x^\star)$. On the other hand, since $y^\star \in \partial_\Phi g(x^\star)$, from \cref{thm:phi_subgradients} we have that $g(x^\star) + g^\Phi(y^\star) - \Phi(x^\star, y^\star) = 0$, which is the claimed result.
\end{proof}
\begin{remark}
    Note that since the dual cost $G$ is a true difference of $\Phi$-convex function, it is straightforward that the dual gap function $\gapg$ enjoys the properties of $\gapf$ described in \cref{thm:gaps}. We utilize both gap functions in our subsequent analysis, in order to strictly characterize the decrease of the function values.
\end{remark}

Now that we have established the desired properties of the gap function, we proceed to the convergence results of the proposed method. First, we state and prove the following sufficient decrease property:
\begin{lemma}[sufficient decrease] \label{thm:sufficient_descent} 
Let $\{(x^{k}, y^{k})\}_{k\in \bN_0}$ be the iterates generated by the algorithm \eqref{eq:DCA}.
Then the following sufficient decrease property holds true for all $k \in \bN_0$:
    \begin{equation} \label{sufficient_descrease}
        F(x^{k+1}) = F(x^k) - (\gapf(x^k, y^k) + \gapg(y^k, x^{k+1})),
    \end{equation}
    i.e. the sequence of function values is nonincreasing.
\end{lemma}
\begin{proof}
We have the following
    \begin{align*}
        \gapf(x^k, y^k) + \gapg(y^k, x^{k+1}) 
        &=
        g(x^k) + g^\Phi(y^k) - \Phi(x^k, y^k) + f(x^{k+1}) + f^\Phi(y^k) - \Phi(x^{k+1}, y^k)
        \\
        &=
        g(x^k) - g(x^{k+1}) + f(x^{k+1}) - f(x^k)
        \\
        &= F(x^k) - F(x^{k+1}),
    \end{align*}
    where the second equality follows by the fact that $x^{k+1} \in \argmin_{x \in X} g(x) - \Phi(x,y^k)$, which implies that $g^\Phi(y^k) = \Phi(x^{k+1}, y^k) - g(x^{k+1})$ and $f^\Phi(y^k) - \Phi(x^k, y^k) = f^{\Phi \Phi}(x^k) = f(x^k)$ by the $\Phi$-convexity of $f$.
\end{proof}
It is important to note that the primal and dual gap function analysis allows us to not only obtain a tight characterization of the decrease of the function values but also to go beyond the standard majorization-minimization tools. Indeed, the proof of \cref{thm:sufficient_descent} is solely based on the notions of (generalized) convexity: $\Phi$-conjugates and the $\Phi$-Fenchel--Young theorem. The decrease of the function values in one step of the algorithm is illustrated in \cref{fig:suff_dec_gap}.
\begin{figure}[!h] 
\centering
\begin{minipage}{.5\textwidth}
  \centering
  \includegraphics[]{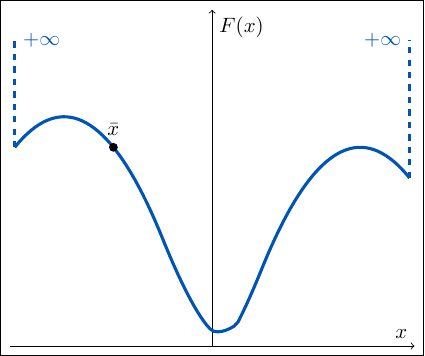}
  \label{fig:sub1}
\end{minipage}%
\begin{minipage}{.5\textwidth}
  \centering
  \includegraphics[]{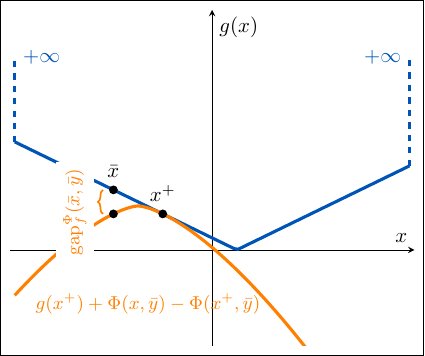}
  \label{fig:sub2}
\end{minipage}%
\vskip\baselineskip
\begin{minipage}{.5\textwidth}
  \centering
  \includegraphics[]{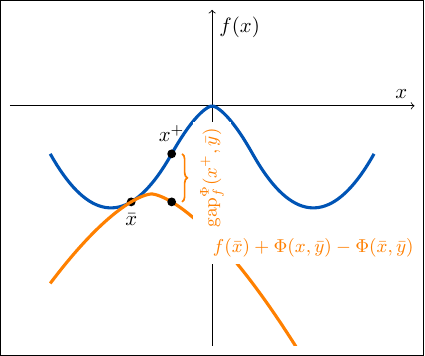}
  \label{fig:sub3}
\end{minipage}%
\begin{minipage}{.5\textwidth}
  \centering
  \includegraphics[]{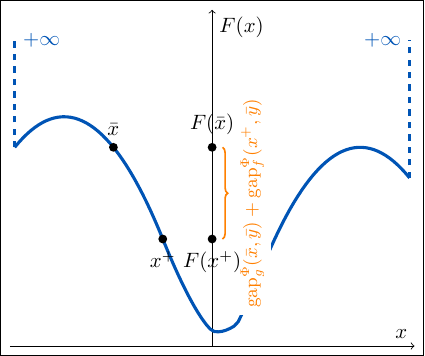}
  \label{fig:sub4}
\end{minipage}%
\caption{Illustration of the $\Phi$-DC decomposition of the function $F(x) := g(x)-f(x)$ with $g(x):=\tfrac{1}{2}|x-\tfrac{1}{2}| + \delta_{[-4,4]}(x)$, $f(x):=-|x|^{1.5} \text{ if } x \in (-1, 1), \tfrac{1}{2}x^2-\tfrac{5}{2}|x|+1 \text{ otherwise}$ and $\Phi(x,y)=-\tfrac{3}{2}|\tfrac{x-y}{2}|^{1.5}$ and application of the proposed scheme: $\bar x$ is a point in $X$ and $x^+$ is the corresponding $\Phi$-DCA step. Starting from the top left corner and continuing clockwise: the function $F$; $g$ along with its $\Phi$-minorant; function $f$ along with its $\Phi$-minorant; the sum of both gap functions as the decrease of $F$.}\label{fig:suff_dec_gap}
\end{figure}

Next we utilize the decrease of the function values as well as the lower boundedness of $F$ in order to show that the sum of the gap functions goes to $0$ and as a consequence, that every limit point $(x^\star, y^\star)$ of $\{(x^k, y^k)\}_{k \in \bN_0}$ is such that $x^\star$ is a $\Phi$-critical point.

\begin{theorem}
\label{thm:asymptotic_convergence}
    Let $\{(x^{k}, y^{k})\}_{k \in \bN_0}$ be the sequence of iterates generated by \eqref{eq:DCA}. Then,
    \begin{align} \label{eq:limit_convergence}
        \lim_{k\to \infty} \gapf(x^k, y^k) = 0 \text{ and } \lim_{k\to \infty} \gapg(y^k, x^{k+1}) = 0.
    \end{align}
    Moreover, every limit point $(x^\star, y^\star)$ with $(x^\star, y^\star) \in X \times Y$ of the sequence of iterates satisfies $y^\star \in \partial_\Phi f(x^\star) \cap \partial_\Phi g(x^\star)$, i.e. $x^\star$ is a $\Phi$-critical point of $F$.
\end{theorem}
\begin{proof}
    Thanks to \cref{thm:sufficient_descent} we have for all $k \geq 0$:
    \begin{equation*}
        F(x^{k+1}) = F(x^k) - (\gapf(x^k, y^k) + \gapg(y^k, x^{k+1})).
    \end{equation*}
From \cref{assum:lower_bound}, we know that $-\infty < \alpha \leq F(x^{K})$ for any $k \in \bN_0$ and so summing the equality above, we obtain:
\begin{align}
-\infty < \inf F - F(x^0) &\leq F(x^{K}) -F(x^0) \notag
\\
&= \sum_{k=0}^{K-1} F(x^{k+1}) -F(x^k)\leq -\sum_{k=0}^{K-1}[\gapf(x^k, y^k) + \gapg(y^k, x^{k+1})] \label{eq:inequality_telescope}.
\end{align}
Thus, we obtain the following inequality for the sum of the gap functions:
\begin{align} \label{eq:gap_sum}
    \sum_{k=0}^{K-1}[\gapf(x^k, y^k)+ \gapg(y^k, x^{k+1})] \leq F(x^0) - \inf F,
\end{align}
which in turn implies that $\{\sum_{k=0}^{K-1}[ \gapf(x^k, y^k)+ \gapg(y^k, x^{k+1})] \}_{K \in \bN_0}$ is bounded and thus \eqref{eq:limit_convergence} follows, since $\gapf(x^k, y^k) \geq 0$ and $\gapg(y^k, x^{k+1}) \geq 0$ for all $k \in \bN_0$.

Now, let $(x^\star,y^\star)$ be a limit point of the sequence of iterates and $(x^{k_j},y^{k_j}) \to (x^\star,y^\star)$ be a corresponding subsequence.
Since $\gapf$ is lsc we have that
$$
0 \leq \gapf(x^\star, y^\star) \leq \lim_{j \to \infty} \gapf(x^{k_j},y^{k_j}) = 0,
$$
where the equality follows from \eqref{eq:limit_convergence}. Since $f$ is lsc on $X$ and $\Phi$ is continuous on $X \times Y$ and $y^{k_j} \in \partial_\Phi f(x^{k_j})$, we have that
\begin{equation} \label{eq:phi_subg_osc}
    f(x) \geq \lim_{j \rightarrow \infty} f(x^{k_j}) + \Phi(x, y^{k_j}) - \Phi(x^{k_j}, y^{k_j}) \qquad \forall x\in X,
\end{equation}
and thus $f(x) \geq f(x^\star) + \Phi(x, y^\star) -\Phi(x^\star, y^\star)$, which by the definition of the $\Phi$-subgradient in turn means that $y^\star \in \partial_\Phi f(x^\star)$. Finally, we use \cref{thm:gaps} to obtain the desired result.
\end{proof}
Note that the reasoning in the last lines of the proof of \cref{thm:asymptotic_convergence} and specifically inequality \eqref{eq:phi_subg_osc} shows that the $\Phi$-subdifferential is an outer semicontinuous mapping in the setting of our paper, as by choosing any $x^k \to \bar x \in X$ and $y^k \to \bar y \in Y$ such that $y^k \in \partial_\Phi f(x^k)$, we have that $\bar y \in \partial_\Phi f(\bar x)$.
\begin{remark}
    Inequality \eqref{eq:gap_sum} is important to our analysis, since it can be used to immediately obtain the convergence rates of well-known algorithms that are special instances of the $\Phi$-DCA. In order to illustrate this fact, let us switch again to the setting of \cref{example:PGM} and assume moreover that $g$ is convex. Then, we have that
    \begin{align*}
        \gapf(x^k, x^k + \tfrac{1}{L}\nabla f(x^k)) = \langle \nabla f(x^k),x^{k+1}-x^k \rangle - \tfrac{L}{2}\|x^{k+1}-x^k\|^2 + g(x^k)-g(x^{k+1})
    \end{align*}
    Using the optimality conditions for the update of $x^{k+1}$ and the convexity of $g$ between points $x^k$ and $x^{k+1}$ we obtain:
    \begin{align*}
        \gapf(x^k, x^k- \tfrac{1}{L}\nabla f(x^k)) &\geq \langle \nabla f(x^k),x^{k+1}-x^k \rangle - \tfrac{L}{2}\|x^{k+1}-x^k\|^2 
        \\
        & \langle \nabla f(x^k) + L(x^k-x^{k+1}),x^k-x^{k+1} \rangle
        \\
        &= \tfrac{L}{2}\|x^{k+1}-x^k\|^2
    \end{align*}
    Then, \eqref{eq:gap_sum} readily implies that
    \begin{equation*}
        \min_{i = 1,\dots,K}\|x^{i+1}-x^i\|^2 \leq \tfrac{2(F(x^0)-\inf F)}{L (K+1)},
    \end{equation*}
    which is a classical result for the convergence of the proximal gradient algorithm in the nonconvex setting for stepsize $\tfrac{1}{L}$.
\end{remark}
Next we specialize the previous result to the case where $X = \bR^n$ and $F$ is a level bounded function, which are standard assumptions for subsequential convergence in the nonconvex setting. Moreover we assume that $\partial_\Phi f$ is bounded, which is a mild requirement. 
\begin{corollary}
    Suppose that $X = \bR^n $ and $ Y = \bR^m$ and that $F$ is level bounded. Assume moreover that $\partial_\Phi f$ is locally bounded on $\bR^n$. Then, the set of limit points of $\{x^k\}_{k \in \bN_0}$ is nonempty and every limit point $x^\star$ is a $\Phi$-critical point.
\end{corollary}
\begin{proof}
    First note that from the decrease \cref{thm:sufficient_descent} of the function values along with the level boundedness of $F$, $\{x^k\}_{k \in \bN_0}$ is a bounded sequence and as such the set of limit points is nonempty. Now, consider a subsequence $x^{k_j} \to x^\star$. Since $\partial_\Phi f$ is locally bounded and $\{x^k\}_{k \in \bN_0}$ is bounded, $\{y^{k_j}\}_{j \in \bN_0}$ is also bounded in light of \cite[Proposition 5.15]{RoWe98}. Therefore we can assume, up to extracting a subsequence, that $y^{k_j} \to y^\star \in Y$. The claimed result now follows from \cref{thm:asymptotic_convergence}.
\end{proof}

\subsection{Sublinear convergence rate analysis}
Our analysis facilitates a $\Phi$-Bregman proximal point interpretation of the algorithm which we provide next: in the case where the $\Phi$-subdifferential of $f$ admits a continuous selection along the iterates of the algorithm, the method can be viewed as a $\Phi$-Bregman proximal point algorithm. 

More precisely, for this subsection we assume that 
there exists a continuous selection of $\partial_\Phi f$ on $\intr X$, which we will denote by $\phif$. Therefore, we have that $\phif : X \to Y$ is such that $\phif(x) \in \partial_\Phi f(x)$ for all $x \in \intr X$. Then, we can define the $\Phi$-Bregman divergence generated by $\Phi$-convex $f$ between two points as follows:
\begin{equation} \label{eq:phibreg}
    \phibreg(x, \bar x) := f(x) - f(\bar x) - \Phi(x, \phif(\bar x)) + \Phi(\bar x, \phif(\bar x)) \qquad \forall x \in X, \intr X 
\end{equation}
It is straightforward that this quantity is the dual gap function $\gapg$ of the $\Phi$-DCA algorithm, composed with $\phif$ and so from the $\Phi$-Fenchel-Young inequality we have that $\phibreg(x, \bar x) \geq 0$. With this insight, we can rewrite \eqref{eq:majorize_minimize} as a $\Phi$-Bregman proximal point method in the following way:
\begin{equation} \label{eq:phi_breg_iter}
    x^{k+1} \in \argmin_{x \in X}F(x) + \phibreg(x, x^k).
\end{equation}
Note that as formally shown in \cref{lemma:bgm_linear}, when $\Phi(x, y) = -L D_h(x,y)$ and $-f$ is $L$-smooth relative to $h$, the $\Phi$-Bregman divergence becomes the classical Bregman divergence generated by $Lh + f$.

\begin{remark}
In contrast to classical methods, we do not assume that $\phibreg(x, \bar x) > 0$ for $x \neq \bar x$. In most of the examples of \cref{sec:examples}, such an inequality could be imposed by considering a smaller stepsize. Let $-f$ be $L$-smooth on $\bR^n$ and $\Phi(x,y) = -\tfrac{1}{2\gamma}\|x-y\|^2$. Then, for $\gamma \leq \tfrac{1}{L}$, $-f$ has a nonempty and single-valued $\Phi$-subdifferential for all $x \in \bR^n$ and from the standard Euclidean descent lemma we have that
\begin{align*}
        -f(x) 
        &\leq -f(\bar x) - \langle \nabla f(\bar x),x-\bar x \rangle + \tfrac{L}{2}\|x-\bar x\|^2
        \\
        &= -f(\bar x) - \langle \nabla f(\bar x),x-\bar x \rangle + \tfrac{1}{2\gamma}\|x-\bar x\|^2 - \tfrac{1-\gamma L}{2}\|x-\bar x\|^2,
\end{align*}
for all $x \in \bR^n$. With some algebraic manipulations (see also \cref{lemma:pgm_linear:gaps}) we can write $\phibreg(x, \bar x) = f(x) -f(\bar x) - \langle \nabla f(\bar x),x-\bar x \rangle + \tfrac{1}{2\gamma}\|x-\bar x\|^2$ and as such from the inequality above we obtain $\phibreg(x, \bar x) \geq \tfrac{1-\gamma L}{2}\|x-\bar x\|^2$, which in turn implies that $\phibreg(x, \bar x) > 0$ for $x \neq \bar x$ and $\gamma < \tfrac{1}{L}$. 
\end{remark}

We next state the assumptions we consider for the remainder of this subsection:
\begin{assumenum}[resume]
    \item $X$ is closed and convex. \label{assum:convex_set}
    \item There exists a continuous selection 
    $\phif$ of $\partial_\Phi f$ on $\intr X$ and $\ran (\partial_\Phi g)^{-1} \subseteq \intr X$. \label{assum:cont_sel}
\end{assumenum}
\Cref{assum:convex_set} is a classical assumption in the field of convex optimization while \cref{assum:cont_sel} describes the existence of a continuous selection of $\partial_\Phi f$ and enforces the iterates of the algorithm to stay in the interior of $X$ the latter being a standard assumption in the related literature \cite{bolte2018first, bauschke2017descent}.

This viewpoint allows us to obtain a sublinear convergence rate for the proposed algorithm under some generally mild conditions. First we analyze the convergence of the algorithm under a generalization of the classical three-point property of Bregman divergences. A similar condition was also considered in \cite{léger2023gradient} to in order to obtain a sublinear rate.
\begin{theorem} \label{thm:phi_ppa_three_point}
    Let $\{x^k\}_{k \in \bN_0}$ be the sequence of iterates generated by \eqref{eq:phi_breg_iter} and \cref{assum:convex_set,assum:cont_sel} hold true. Let moreover $F$ be convex and $f$ and $\Phi(\cdot,y)$ be differentiable on $\intr X$. 
    If the following three-point property holds:
        \begin{equation}
            \langle \nabla f(x^{k+1}) - \nabla_x \Phi(x^{k+1}, \phif(x^k)), x-x^{k+1} \rangle \leq \phibreg(x, x^{k})-\phibreg(x, x^{k+1})-\phibreg(x^{k+1}, x^k),
        \end{equation}
        for $x \in X$, then for all $K \geq 1$:
        \begin{equation}
            F(x^K) - F(x) \leq \tfrac{\phibreg(x, x^0)}{K}.
        \end{equation}
\end{theorem}
\begin{proof}    
    Consider $v^{k+1} \in \partial F(x^{k+1})$ and the variational inequality for the solution of the inner minimization problem in \eqref{eq:phi_breg_iter}:
    \begin{equation*}
        \langle v^{k+1} + \nabla_x \phibreg(x^{k+1}, x^k), x-x^{k+1}\rangle \geq 0,
    \end{equation*}
    for any $x \in X$. By the convex subgradient inequality for $F$ between points $x \in X$ and $x^{k+1}$ we can further bound
    \begin{equation*}
        \langle \nabla_x \phibreg(x^{k+1}, x^k), x-x^{k+1}\rangle \geq F(x^{k+1})-F(x).
    \end{equation*}
    In light of the assumption of the theorem, the above inequality further means that
    \begin{equation} \label{eq:rate_ineq_ppa}
        F(x^{k+1})-F(x) \leq \phibreg(x, x^k) - \phibreg(x, x^{k+1}) - \phibreg(x^{k+1}, x^k)
    \end{equation}
    Summing \eqref{eq:rate_ineq_ppa} from $k = 0$ to $K \in \bN$ we get
    \begin{equation*}
        \sum_{k=0}^K F(x^{k+1})-F(x) \leq \phibreg(x, x^0).
    \end{equation*}
    Now, since the sequence $F(x^k)$ is nonincreasing we can further bound the inequality above as
    \begin{equation*}
        F(x^{K})-F(x) \leq \tfrac{\phibreg(x, x^0)}{K},
    \end{equation*}
    which is the claimed result.
\end{proof}
\begin{remark}
    At first glance, the three-point property assumed in \cref{thm:phi_ppa_three_point} might seem unintuitive. However, it is important to note that it is a natural extension of the well-known three-point property of the Bregman divergences \cite[p. 4]{bolte2018first}. As such, it is straightforward to demonstrate that our three-point property holds for these types of couplings. This directly implies that our analysis successfully recovers the classical sublinear convergence rate of the Bregman proximal gradient method when applied to convex problems. Furthermore, \cref{thm:phi_ppa_three_point} also recovers the sublinear rate analysis for the standard DCA in the case where both $f$ and $g$ are smooth and moreover $F$ is convex, a result which was first stated in \cite[Corollary 1]{faust2023bregman}.
\end{remark}
A sublinear convergence rate can also be achieved when a different property than the (extension of) the three-point property holds. This is captured in the terms of subhomogeneity of a function, adapted from \cite{aze1995uniformly}:
\begin{definition}[subhomogeneity of order $\nu$]
    Let $\omega:\bR^n \to \bR_+$ be convex and $\nu \geq 1$. If $\omega(\theta x) \leq \theta^{\nu} \omega(x)$ for any $\theta \in [0,1]$ we say that $\phi$ is subhomogeneous of order $\nu$. In particular this means that $\omega(0)=0$.
\end{definition}
\begin{remark}
    Every convex function $\omega:\bR^n \to \bR$ with $\omega(0) = 0$, is subhomogeneous of order $\nu = 1$. This follows directly from the convexity inequality:
    \begin{equation*}
        \omega(\theta x) = \omega((1-\theta)0 + \theta x) \leq \theta \omega(x).
    \end{equation*}
    The most important type of functions that are subhomogeneous of order $\nu$ are the norms to some power $\nu$, $\|\cdot\|^\nu$, which are actually homogeneous, $\|\theta x\|^\nu = \theta^\nu \|x\|^\nu$.
\end{remark}
In this setting, we can modify the $\Phi$-DCA in order to incorporate an averaging scheme (see for example \cite[Algorithm 3]{doikov2020inexact}). For $x^0 \in \intr X$, the $\Phi$-DCA becomes:
\begin{align} \label{eq:DCA_averaging}
\begin{cases}
    w^k = \lambda_k x^k + (1-\lambda_k )x^0 \\
    y^k = \phif (w^k) \\
    x^{k+1} \in (\partial_\Phi g)^{-1}(y^k),
\end{cases}
\end{align}
where $\{\lambda_k\}_{k \in \bN_0}$ is a positive sequence with $\lambda_0 = 0$ and $\lambda_k \leq 1$. Note that by choosing $\lambda_k = 1$ we get \eqref{eq:DCA}. Since $x^0 \in \intr X$ and $\lambda_k \leq 1$ we have that $w^k \in \intr X$ and as such the algorithm is well-defined. Once again, the scheme can be interpreted as a $\Phi$-Bregman proximal point method with the following update:
\begin{equation} \label{eq:phibreg_averaging}
    x^{k+1} \in \argmin_{x \in X}F(x) + \phibreg(x, w^k).
\end{equation}

\begin{theorem} \label{thm:phi_ppa_sub}
    Let $x^\star \in \argmin F$ and \cref{assum:convex_set,assum:cont_sel} hold true and assume that $\phibreg(x, w^k) \leq \omega(w^k - x)$ for $x \in X$ and some $\omega:\bR^n \to \bR_+$ subhomogeneous of order $p+1$ with $p > 0$. Then, for the sequence $\{x^k\}_{k \in \bN_0}$ generated by \eqref{eq:DCA_averaging} the following statements hold true:
    \begin{thmenum}
        \item If $\lambda_k = 1$, then $F(x^K) - F(x^\star) \leq \mathcal{D}_0 \frac{(p+1)^{p+1}}{K^{p}}$, where $\mathcal{D}_0=\sup\{ \omega(x - x^\star) : F(x) \leq F(x^0)\}$. \label{thm:phi_ppa_sub:rate_simple}

        \item If $\lambda_k = (\tfrac{k}{k+1})^{p+1}$, then $F(x^K) - F(x^\star) \leq  \frac{(p+1)^{p+1}\omega(x^0-x^\star)}{K^{p}}$
        \label{thm:phi_ppa_sub:rate_averaging}
    \end{thmenum}
\end{theorem}
\begin{proof}
    To begin with, in light of \eqref{eq:phibreg_averaging} we have that
    \begin{equation*}
        F(x^{k+1}) + \phibreg(x^{k+1}, w^k) \leq F(x) + \phibreg(x, w^k),
    \end{equation*}
    which from the assumption of the theorem implies that
    \begin{equation}
        F(x^{k+1}) \leq F(x) + \omega(w^k - x),
    \end{equation}
    for $x \in X$. Now consider the increasing sequence $\{A_k\}_{k=0}^\infty$ with $A_k  = k^{p+1}$ and $A_0=0$. We denote by $a_{k+1} := A_{k+1} - A_k$. Plugging $x := \frac{a_{k+1} x^\star + A_k x^k}{A_{k+1}}$, which as the convex combination of $x^\star$ and $x^k$ is in $X$, in the inequality above we obtain by convexity of $F$,
    \begin{equation}\label{eq:phi_ppa_dec}
        F(x^{k+1}) \leq \tfrac{a_{k+1}}{A_{k+1}}F(x^\star) + \tfrac{A_{k}}{A_{k+1}} F(x^k) + \omega(w^k-x).
    \end{equation}
    Now we treat the two items of the theorem. Let $\theta_k:=\tfrac{a_{k+1}}{A_{k+1}} < 1$.

    ``\labelcref{thm:phi_ppa_sub:rate_simple}'':
     We have that $w^k-x = x^k - \frac{a_{k+1} x^\star + A_k x^k}{A_{k+1}} = \tfrac{\alpha_{k+1}}{A_{k+1}}(x^k-x^\star)$ and from the subhomogeneity of $\omega$:
    \begin{align}
        \omega(\theta_k x) \leq \theta_k^{p+1} \omega(x),
    \end{align}
    and hence
    \begin{align*}
        F(x^{k+1}) \leq \tfrac{a_{k+1}}{A_{k+1}}F(x^\star) + \tfrac{A_{k}}{A_{k+1}}F(x^k) + \theta_k^{p+1} \omega(x^k-x^\star)
    \end{align*}
    Multiplying both sides with $A_{k+1}$ we get since $a_{k+1}= A_{k+1} - A_k$
    \begin{align*}
        A_{k+1}\big(F(x^{k+1})-F(x^\star)\big) \leq A_{k}\big(F(x^k)-F(x^\star)\big) + \tfrac{a_{k+1}^{p+1}}{A_{k+1}^{p}}\omega (x^k - x^\star).
    \end{align*}
    Summing the inequality from $k=0$ to $k=K-1$ we obtain since $A_0 = 0$:
    \begin{align} \label{eq:decrease_inexact_ppa_main}
        A_{K}\big(F(x^{K}) - F(x^\star)\big) \leq \sum_{k=0}^{K-1} \tfrac{a_{k+1}^{p+1}}{A_{k+1}^{p}}\omega (x^k - x^\star).
    \end{align}
    In light of \cref{thm:sufficient_descent}, we have that $F(x^{k+1}) \leq F(x^0)$ for all $k \geq 0$ and hence $\omega(x^K - x^\star) \leq \mathcal{D}_0$ for any $K \geq 0$. Thus we can further bound \cref{eq:decrease_inexact_ppa_main}:
    \begin{align*} 
        A_{K}\big(F(x^{K}) - F(x^\star)\big) \leq \mathcal{D}_0\sum_{k=0}^{K-1} \tfrac{a_{k+1}^{p+1}}{A_{k+1}^{p}}.
    \end{align*}
    Using the fact that $\sum_{k=1}^{K} \tfrac{a_{k}^{p+1}}{A_{k}^{p}} \leq (p+1)^{p+1}K$ \cite[Equation (35)]{doikov2020inexact}:
    \begin{align*} 
        A_{K}\big(F(x^{K}) - F(x^\star)\big) \leq \mathcal{D}_0 (p+1)^{p+1}K.
    \end{align*}
    Dividing by $A_K$ we obtain:
    \[
        F(x^{K}) - F(x^\star) \leq \mathcal{D}_0 \frac{(p+1)^{p}}{K^{p}}.
    \] 
    ``\labelcref{thm:phi_ppa_sub:rate_averaging}'': We have that 
    $$w^k - x = \lambda_k x^k + (1-\lambda_k) x^0 - (1-\lambda_k)x^\star - \lambda_k x^k = (1-\lambda_k)(x^0 - x^\star) = \theta_k (x^0-x^\star).
    $$
    Substituting in \eqref{eq:phi_ppa_dec} and using the same algebraic manipulations as in the proof of \labelcref{thm:phi_ppa_sub:rate_simple}, we obtain for any $K \geq 0$:
    \begin{equation*}
        A_{K}\big(F(x^{K}) - F(x^\star)\big) \leq \sum_{k=0}^{K-1} \tfrac{a_{k+1}^{p+1}}{A_{k+1}^{p}}\omega (x^0 - x^\star) \leq (p+1)^{p+1}K \omega(x^0-x^\star).
    \end{equation*}
    Dividing by $A_k$ we obtain the claimed result.
\end{proof}

The result described in \cref{thm:phi_ppa_sub} incorporates the convergence rates of various splitting methods. This is demonstrated in the following corollaries, which expand upon \cref{example:holder} and \cref{example:tensor}. The first one describes the convergence of the H\"older proximal gradient method, retrieving the results of \cite{bredies2008forward}.
\begin{corollary}[convergence of H-PGM] \label{cor:hoelder}
    Let $\{x^k\}_{k \in \bN_0}$ be the sequence of iterates generated by \eqref{eq:DCA_averaging} in the setting of \cref{example:holder}. If $-f$ is convex and for $\phif(x) = (x, \nabla f(x))$, we have that:
    \begin{equation*}
        \phibreg(x, x^k) = f(x) - f(x^k) - \langle \nabla f(x^k) ,x-x^k \rangle + \tfrac{H}{\nu+1}\|x-x^k\|^{\nu+1} \leq \tfrac{H}{\nu+1}\|x-x^k\|^{\nu+1}.
    \end{equation*}
    If moreover, $F$ is convex, then the following statements hold true for all $K\geq 1$:
    \begin{corenum}
        \item  If $\lambda_k = 1$, then $F(x^K) - F(x^\star) \leq  \frac{(\nu+1)^{\nu+1}\mathcal{D}_0}{K^{p}}$ 
        for $\mathcal{D}_0=\tfrac{H}{\nu+1} \sup\{ \|x - x^\star\|^{\nu+1} : F(x) \leq F(x^0)\}$; \label{cor:hoelder:simple}

        \item If $\lambda_k = (\tfrac{k}{k+1})^{\nu+1}$, then $F(x^K) - F(x^\star) \leq  \frac{(\nu+1)^{\nu}H \|x^0-x^\star\|^{\nu+1}}{K^{\nu}}$, \label{cor:hoelder:averaging}
    \end{corenum}
    where $x^\star \in \argmin_{x \in \bR^n}F(x)$.
\end{corollary}
\begin{proof}
    The form of $\phibreg$ follows directly by substituting $\phif$ into \eqref{eq:phibreg}. Regarding the upper bound on $\phibreg$, we have that by convexity of $-f$:
    \begin{equation*}
        -f(x) \geq -f(x^k) - \langle \nabla f(x^k),x-x^k \rangle \qquad \forall x \in \bR^n,
    \end{equation*}
    which means that $f(x) - f(x^k) - \langle \nabla f(x^k) ,x-x^k \rangle \leq 0$ and thus the claimed bound holds. Then, \labelcref{cor:hoelder:simple} and \labelcref{cor:hoelder:averaging} follow by \cref{thm:phi_ppa_sub:rate_simple} and \cref{thm:phi_ppa_sub:rate_averaging} for $\omega = \tfrac{H}{\nu+1}\|\cdot\|^{\nu+1}$.
\end{proof}
The following corollary retrieves the convergence rate results of the exact version of \cite[Algorithm 1]{doikov2020inexact}.
\begin{corollary}[convergence of TM]\label{cor:tensor}
    Let $\{x^k\}_{k \in \bN_0}$ be the sequence of iterates generated by \eqref{eq:DCA_averaging} in the setting of \cref{example:tensor}. For $\phif(x) = (x, D^1 f(x), \dots, D^p f(x))$, we have that:
    \begin{equation*}
        \phibreg(x, x^k) 
        = f(x) - f(x^k) - \sum_{i=1}^p \tfrac{1}{i!}D^p f(x^k)[x-x^k]^i + \tfrac{L_p}{(p+1)!}\|x-x^k\|^{p+1} \leq \tfrac{2L_p}{(p+1)!}\|x-x^k\|^{p+1}.
    \end{equation*}
    If moreover, $F$ is convex, then the following statements hold true for all $K\geq 1$:
    \begin{corenum}
        \item  If $\lambda_k = 1$, then $F(x^K) - F(x^\star) \leq \frac{(p+1)^{p+1}\mathcal{D}_0}{K^{p}}$ 
        for $\mathcal{D}_0=\tfrac{2L_p}{p+1} \sup\{ \|x - x^\star\|^{p+1} : F(x) \leq F(x^0)\}$; \label{cor:tensor:simple}

        \item If $\lambda_k = (\tfrac{k}{k+1})^{p+1}$, then $F(x^K) - F(x^\star) \leq  \frac{2(p+1)^{p}L_p \|x^0-x^\star\|^{p+1}}{K^{p}}$, \label{cor:tensor:averaging}
    \end{corenum}
    where $x^\star \in \argmin_{x \in \bR^n}F(x)$.
\end{corollary}
\begin{proof}
    The form of $\phibreg$ follows directly by substituting $\phif$ into \eqref{eq:phibreg}. Regarding the upper bound on $\phibreg$, we have that 
    \begin{align*}
        \phibreg(x, x^k) \leq \left|f(x) - f(x^k) - \sum_{i=1}^p \tfrac{1}{i!}D^p f(x^k)[x-x^k]^i\right| + \tfrac{L_p}{(p+1)!}\|x-x^k\|^{p+1} \leq \tfrac{2L_p}{(p+1)!}\|x-x^k\|^{p+1},
    \end{align*}
    where the inequality follows by the bound on the difference between the function and the Taylor approximation \cite[Equation (3)]{doikov2020inexact}. The convergence rate results, \labelcref{cor:tensor:simple} and \labelcref{cor:tensor:averaging} follow by \cref{thm:phi_ppa_sub:rate_simple} and \cref{thm:phi_ppa_sub:rate_averaging} for $\omega = \tfrac{2L_p}{p+1}\|\cdot\|^{p+1}$
\end{proof}
\Cref{thm:phi_ppa_sub} also leads to new convergence rate results for the recently introduced a-PGM discussed in \cref{example:apgm}:
\begin{corollary}\label{cor:anisotropic}
    Let $\{x^k\}_{k \in \bN_0}$ be the sequence of iterates generated by \eqref{eq:DCA_averaging} in the setting of \cref{example:apgm} and $h$ be $L_h$-smooth on $\bR^n$. If $-f$ and $g$ are convex, then the following statements hold true for all $K\geq 1$:
    \begin{corenum}
        \item If $\lambda_k = 1$, then $F(x^K) - F(x^\star) \leq \frac{2 \mathcal{D}_0}{K}$ 
        for $\mathcal{D}_0=LL_h \sup\{ \|x - x^\star\|^{2} : F(x) \leq F(x^0)\}$; \label{cor:anisotropic:simple}
        \item If $\lambda_k = (\tfrac{k}{k+1})^{2}$, then $F(x^K) - F(x^\star) \leq  \frac{2 LL_h \|x^0-x^\star\|^2}{K}$, \label{cor:anisotropic:averaging}
    \end{corenum}
    where $x^\star \in \argmin_{x \in \bR^n}F(x)$.
\end{corollary}
\begin{proof}
    We have the following:
    \begin{align*}
        \phibreg(x, x^k) = f(x) - f(x^k) - \tfrac{1}{L} \star h(x^k - y^k) + \tfrac{1}{L} \star h(x-y^k),
    \end{align*}
    and using the $L_h$-smoothness of $h$ we can further bound:
    \begin{align*}
        \phibreg(x, x^k) &\leq f(x) - f(x^k) - \langle \nabla f(x^k),x-x^k \rangle + \tfrac{L_h L}{2}\|x-x^k\|^2
        \\
        & \leq \tfrac{L_h L}{2}\|x-x^k\|^2,
    \end{align*}
    where the second inequality follows from the convexity of $-f$. The convergence rate results follow from \cref{thm:phi_ppa_sub:rate_simple} and \cref{thm:phi_ppa_sub:rate_averaging}.
\end{proof}
An example of a function satisfying the assumptions of \cref{cor:anisotropic} is $h(x) = \sum_{i=1}^n 2\ln(1+\exp(x_i))-x_i$, which is studied in \cite[Example 4.10]{laude2022anisotropic}. Showing that $h$ is $1$-smooth follows directly from the fact that for $\tilde{h}(x) = 2\ln(1+\exp(x))-x, x \in \bR$, we have $\tilde{h}^{''}(x) = \tfrac{2\exp(x)}{(\exp(x)+1)^2} \leq 1$. A similar function that also satisfies this condition is $h(x) = \sum_{i=1}^n \ln \cosh (x_i)$.

Note that in the corollaries above, the bound on $\phibreg$ follows by using upper bounds for $f$, in contrast to the asymptotic analysis of the paper where we relied on the lower bounds generated by the $\Phi$-subgradient inequality. In general, the radius of the initial level-set may be unbounded, in which case the obtained convergence rates of \cref{thm:phi_ppa_sub:rate_simple} are not informative. Nevertheless, when $F$ has bounded level-sets, as is the case for example when $\argmin F$ is a bounded set, we obtain the rates from \cite{bredies2008forward} and \cite{doikov2020inexact} in the exact version of the method. Moreover, for the averaging version of the algorithm \eqref{eq:DCA_averaging} we have $\omega(x^0-x^\star)$ in the rate, which in the setting of \cref{example:tensor} amounts to $\tfrac{1}{p+1}\|x^0-x^\star\|^{p+1}$, thus not requiring level-boundedness of $F$, as in \cite[Theorem 5]{doikov2020inexact}.
\begin{remark}
    The conditions in \cref{thm:phi_ppa_three_point,thm:phi_ppa_sub} are complementary in the following sense: consider the case of the H\"older proximal gradient method and notice that indeed the extension of the three-point property that we consider in \cref{thm:phi_ppa_three_point} does not hold. Furthermore, note that \cref{thm:phi_ppa_three_point} gives convergence rates of $1/k$ and thus cannot encapsulate the faster convergence rate of e.g. the tensor methods.
\end{remark}
All of the examples presented thus far belonged to the forward-backward splitting scheme. Nevertheless, our analysis also applies to backward schemes such as the various instances of the proximal point method (PPM). The following example demonstrates this fact.
\begin{example}[anisotropic proximal point method (a-PPM) under subhomogeneity]
    \label{example:power_ppa}
    Consider the case where $f = 0$, $g$ is convex and $\Phi(x,y) = - h(x-y)$ for some $h$ subhomogeneous of order $p+1$. Then, \eqref{eq:DCA} takes the following form:
    \begin{equation} \label{eq:high_order_ppa}
        x^{k+1} \in \argmin_{x \in \bR^n} g(x) + h(x-x^k),
    \end{equation}
    which is an instance of a-PPM \cite{laude2019optimization}. Note that this update generalizes the power PPM, the exact version of the algorithm analyzed in \cite{nesterov2023inexact, oikonomidis2023global} and it is straightforward that since $f = 0$, $\phibreg(x, x^k) = h(x-x^k)$. Therefore, the convergence rate of the method follows directly by \cref{thm:phi_ppa_sub}. 
\end{example}
The $\Phi$-DCA framework also generalizes some classical duality results between the PPM and gradient descent method to the abstract setting of $\Phi$-convexity. It is well-known that PPM on $g \in \Gamma_L(\bR^n)$ is equivalent to gradient descent on its Moreau envelope $\inf \{g + \tfrac{L}{2}\|\cdot - x\|^2 \}$. Moreover, in \cite[Theorem 3.8]{laude2021lower} it is shown that gradient descent on a $L$-Lipschitz smooth function $-f$ is equivalent to the PPM on the $\inf$-\emph{deconvolution} of $-f$, defined as $-\inf \tfrac{L}{2}\|\cdot - x\|^2 + f$. In our setting, these equivalences can be considered as applying the $\Phi$-DCA to the $\Phi$-DC dual function $G$. In the first case, where $F \equiv g$, $G = -g^\Phi = \inf \{g + \tfrac{1}{2\gamma}\|\cdot - x\|^2 \}$ as shown in \cref{example:quadratic}. In the second case, where $F \equiv -f$, we have that $-\inf \{\tfrac{L}{2}\|\cdot - x\|^2 + f\} = \sup \{-\tfrac{L}{2}\|\cdot - x\|^2 - f\} = f^\Phi = G$.

\section{The \texorpdfstring{$\Phi$}{Φ}-DCA-PL inequality and linear convergence rates}
\label{sec:linear_rate_phi_dca}
In this section, we establish a condition that ensures the $q$-linear convergence of the $\Phi$-DCA. This condition is derived from both the primal and dual regularized gap functions introduced earlier in our analysis. 

We formalize this in what we term the \emph{$\Phi$-DCA-Polyak-Łojasiewicz} ($\Phi$-DCA-PL) inequality, which extends the proximal-PL inequality that was studied in \cite{karimi2016linear} and further generalized in \cite{laude2022anisotropic}.
\begin{definition}[$\Phi$-DCA-PL inequality] \label{def:dca_pl}
    We say that $F = g - f$ satisfies the $\Phi$-DCA-PL inequality relative to the coupling function $\Phi$ if the following condition holds true:
    \begin{equation} \label{eq:dca_pl}          
        \gapf(\bar x,\Bar{y}) + \gapg(\bar y, x^+) \geq \mu_1(F(\bar x) - \inf F) + \mu_2(F(x^+) - \inf F),
    \end{equation}
    for $\bar x \in \dom F$, where $\mu_1 + \mu_2 > 0$, $\bar y \in \partial_\Phi f(\bar x)$ and $x^+ \in \argmin_{x \in X} g(x) - \Phi(x, \bar y)$.
\end{definition}
Note that for $\mu_2 = 0$ and $\Phi(x,y) = -\tfrac{L}{2}\|x-y\|^2$ we obtain the proximal-PL inequality from \cite{karimi2016linear}, a fact further demonstrated in \cref{subsubsec:linear_pgm}.

Next we show that under this condition the $q$-linear convergence rate of the function values follows:
\begin{theorem}[$q$-linear rate of the $\Phi$-DCA] \label{thm:clasic_linear_rate}
    Let $\{x^k\}_{k \in \bN_0}$ be the sequence of primal iterates generated by the $\Phi$-DCA \eqref{eq:DCA} and $F$ satisfy the $\Phi$-DCA-PL inequality. Then $\{F(x^k)\}_{k \in \bN_0}$ decreases $q$-linearly and in particular,
    \begin{align} \label{eq:q_linear_rate}
        F(x^{k+1}) - \inf F \leq \frac{1-\mu_1}{1+\mu_2}(F(x^k) - \inf F)
    \end{align}
\end{theorem}
\begin{proof}
    In light of \cref{thm:sufficient_descent} we have that 
    \begin{equation*}
        F(x^{k+1}) = F(x^k) - (\gapf(x^k, y^k) + \gapg(y^k, x^{k+1})).
    \end{equation*}
    By rearranging and using \eqref{eq:dca_pl}, we can further write:
    \begin{equation*}
        F(x^k) - F(x^{k+1}) \geq \mu_1(F(x^k) - \inf F) + \mu_2(F(x^{k+1}) - \inf F)
    \end{equation*}
    and the claimed result follows.
\end{proof}

Many well-established conditions that are known to yield $q$-linear rates can be viewed as specific instances of the $\Phi$-DCA-PL inequality. To illustrate this connection, we explore some of the most significant examples in the following subsections.

\subsection{Linear rate of the standard DCA}
Consider the setting of the standard DCA as in \cref{example:DCA}, where moreover $f$ and $g$ are smooth on $\bR^n$. In this case we have $\Phi(x,y) = \langle x, y \rangle$, while the forward step is $\bar y = \nabla f(\bar x)$ and the backward step $x^+ \in \partial g^*(\nabla f(\bar x))$. We next show that the standard conditions that lead to a $q$-linear convergence rate for the DCA also imply our $\Phi$-DCA-PL inequality. More precisely we will show that the DCA PL inequality introduced in \cite[Definition 1]{faust2023bregman} leads to our \cref{def:dca_pl}. It has the following form:
\begin{align} \label{eq:dca_pl_1}
    & D_{g^*}(\nabla f(\bar x), \nabla g(\bar x)) \geq \eta_1 (F(\bar x)-\inf F) \\
    & D_{f^*}(\nabla g(\bar x), \nabla f(\bar x)) \geq \eta_2 (F(\bar x)-\inf F),\label{eq:dca_pl_2}
\end{align}
where $\eta_1, \eta_2 \geq 0$ are such that $\eta_1 + \eta_2 >0$ and $\bar x \in \dom f$. Note that in \cite[Definition 1]{faust2023bregman} this condition is assumed on a set containing the iterates, but we assume it on $\dom F$ for ease of exposition. Moreover, since $g^*$ and $f^*$ might not be differentiable we consider the respective Bregman distance defined as $D_{g^*}(\nabla f(\bar x), \nabla g(\bar x)) = g^*(\nabla f(\bar x))-g^*(\nabla g(\bar x))-\langle \bar x, \nabla f(\bar x)-\nabla g(\bar x)\rangle$.
\begin{lemma} \label{corr:dca_linear}
    In the setting considered in this subsection, 
    let $\bar x \in \bR^n$ and $x^+ \in \partial g^*(\nabla f(\bar x))$. Then, the following hold:
    \begin{lemenum}
        \item \label{lemma:dca_linear:gaps} The primal and dual gap functions take on the following forms:
        \begin{align*}
        &\gapf(\bar x, \nabla f(\bar x)) 
        = g(\bar x) + g^*(\nabla f(\bar x)) - \langle \bar x,\nabla f(\bar x) \rangle
        = D_g(\bar x, x^+)
        \\
        &\gapg( \nabla f(\bar x), x^+) =
        f(x^+) + f^*(\nabla f(\bar x)) - \langle x^+,\nabla f(\bar x) \rangle
        = D_f(x^+, \bar x),
        \end{align*}
        \item \label{lemma:dca_linear:rate} The DCA PL inequalities \eqref{eq:dca_pl_1} and \eqref{eq:dca_pl_2} imply that $F$ satisfies \cref{def:dca_pl} with $\mu_1 = \eta_1$ and $\mu_2 = \eta_2$.
    \end{lemenum}
\end{lemma}
\begin{proof}
    ``\labelcref{lemma:dca_linear:gaps}'':
    By definition we have that $g^*(\nabla f(\bar x)) = \langle x^+,\nabla f(\bar x) \rangle - g(x^+)$ and as such
    \begin{align*}
        \gapf(\bar x, \nabla f(\bar x)) 
        &= g(\bar x) + g^\star (\nabla f(\bar x)) - \langle \bar x,\nabla f(\bar x) \rangle
        \\
        &= g(\bar x) - g(x^+) - \langle \bar x - x^+, \nabla f(\bar x) \rangle = D_g(\bar x, x^+).
    \end{align*}
    With the same arguments $f^*(\nabla f(\bar x)) = \langle \bar x,\nabla f(\bar x) \rangle - f(\bar x)$ and thus
    \begin{align*}
        \gapg(\nabla f(\bar x), x^+) 
        &= 
        f(x^+) + f^\star(\nabla f(\bar x)) - \langle x^+,\nabla f(\bar x) \rangle
        \\
        &=
        f(x^+) - f(\bar x) - \langle x^+-\bar x, \nabla f(\bar x)\rangle = D_f(x^+, \bar x).
    \end{align*}

    ``\labelcref{lemma:dca_linear:rate}'':
    By the properties of the Bregman distances, we have that $D_g(\bar x, x^+) = D_{g^*}(\nabla g(x^+), \nabla g(\bar x)) = D_{g^*}(\nabla f(\bar x), \nabla g(\bar x))$, since $x^+ \in \partial g^*(\nabla f(\bar x))$. It is thus straightforward that \eqref{eq:dca_pl_1} implies $\gapf(\bar x, \nabla f(\bar x)) \geq \eta_1(F(\bar x) - \inf F)$.

    Consider now \eqref{eq:dca_pl_2} and note that since it holds for $\bar x \in \mathcal{X}$ it also holds for $x^+ \in \partial g^*(\nabla f(\bar x))$:
    \begin{equation*}
        D_{f^*}(\nabla f(\bar x), \nabla f(x^+)) \geq \eta_2 (F(x^+)-\inf F)
    \end{equation*}
    Now, we have that
    \begin{align*}
        D_{f^*}(\nabla f(\bar x), \nabla f(x^+)) &= f^*(\nabla f(\bar x)) - f^*(\nabla f(x^+)) - \langle x^+,\nabla f(\bar x)- \nabla f(x^+)\rangle
        \\
        &= f^*(\nabla f(\bar x)) + f(x^+) - \langle x^+,\nabla f(\bar x)\rangle
        \\
        &= \gapg(\nabla f(\bar x), x^+).
    \end{align*}
    Therefore, \eqref{eq:dca_pl_2} implies $\gapg(\nabla f(\bar x), x^+) \geq \eta_2(F(x^+) - \inf F)$. Summing the two inequalities we obtain the claimed implication.
\end{proof}

\subsection{Linear rate of the proximal gradient method} \label{subsubsec:linear_pgm}
Let us now transfer to the setting of \cref{example:PGM}. As already stressed, the value function in this case is the well-studied Forward-Backward envelope, while the $\Phi$-subgradients of $f$ are the standard gradient steps $\bar y = \bar x + \tfrac{1}{L}\nabla f(\bar x)$. 
A unifying condition for the $q$-linear convergence rate of the algorithm was explored in \cite[Equation (12)]{karimi2016linear}, the so-called proximal-PL inequality: there exists an $\eta > 0$ such that
\begin{equation} \label{eq:schmidt_pl}
    \tfrac{1}{2}\mathcal{D}_g(\bar x, L) \geq \eta(F(\bar x) - \inf F),
\end{equation}
where $\mathcal{D}_g(\bar x, L) = -2L \min_{z \in \bR^n}\{-\langle \nabla f(\bar x),z-\bar x \rangle + \tfrac{L}{2}\|z-\bar x\|^2 + g(z) - g(\bar x)\} = 2L\gapf(\bar x, \bar x + \tfrac{1}{L}\nabla f(\bar x))$.
\begin{lemma} \label{lemma:pgm_linear}
    In the setting considered in this subsection, 
    let $\bar x \in \bR^n$, $\bar y = \bar x + \tfrac{1}{L}\nabla f(\bar x)$ and $x^+ \in \argmin_{x \in \bR^n}g(x)+\tfrac{L}{2}\|x-\bar y\|^2$. Then, the following hold:
    \begin{lemenum}
        \item \label{lemma:pgm_linear:gaps} The primal and dual gap functions take on the following forms:
        \begin{align*}
        &\gapf(\bar x, \bar y) 
        = g(\bar x) - \min_{x\in \bR^n} g(x) + \tfrac{L}{2}\|x - (\bar x + \tfrac{1}{L}\nabla f(\bar x))\|^2 + \tfrac{1}{2L}\|\nabla f(\bar x)\|^2
        \\
        &\gapg(\bar y, x^+) =
        f(x^+) - \min_{x \in \bR^n} \tfrac{L}{2}\|x-\bar y\|^2 + f(x) + \tfrac{L}{2}\|x^+ - \bar y\|^2 = D_{\tfrac{L}{2}\|\cdot\|^2+f}(x^+, \bar x),
        \end{align*}
        \item \label{lemma:pgm_linear:rate} The proximal-PL inequality \eqref{eq:schmidt_pl} implies that $F$ satisfies \cref{def:dca_pl} with $\mu_1 = \tfrac{\eta}{L}$ and $\mu_2 = 0$.
    \end{lemenum}
\end{lemma}
\begin{proof}
    ``\labelcref{lemma:pgm_linear:gaps}'': The form of the primal gap follows by definition along with $g^\Phi(\bar y) = \max_{x \in \bR^n} -\tfrac{L}{2}\|x-\bar y\|^2 - g(x) = -\min_{x \in \bR^n} g(x) + \tfrac{L}{2}\|x-\bar y\|^2$. Regarding the dual gap we have the following:
    \begin{align*}
        \gapg(\bar y, x^+) 
        &=
        f(x^+) - \min_{x \in \bR^n} \tfrac{L}{2}\|x-\bar y\|^2 + f(x) + \tfrac{L}{2}\|x^+ - \bar y\|^2
        \\
        &= 
        f(\bar x) + f(x^+) + \tfrac{L}{2}\|x^+ - \bar y\|^2 - \tfrac{L}{2}\|\bar x - \bar y\|^2
        \\
        &=
        \tfrac{L}{2}\|x^+\|^2 + f(x^+) - (\tfrac{L}{2}\|\bar x\|^2 + f(\bar x))
        - \langle x^+ - \bar x,L\bar x + \nabla f(\bar x) \rangle 
        \\
        &= D_{\tfrac{L}{2}\|\cdot\|^2-f}(x^+, \bar x),
    \end{align*}
    where the first equality follows by the definition of $f^\Phi(\bar y)$ and the second by the fact that $\bar y \in \partial_\Phi f(\bar x)$ and as such $f^\Phi(\bar y) = -\tfrac{L}{2}\|\bar x-\bar y\|^2 + f(\bar x)$ in light of the equivalence in \cref{thm:phi_subgradients}.

    ``\labelcref{lemma:pgm_linear:rate}'': The result is evident from the fact that $\mathcal{D}_g(\bar x, L) = 2L \gapf(\bar x, \bar y)$. 
\end{proof}

\subsection{Linear rate of the Bregman gradient method}
Now we consider the setting of \cref{example:BPGM}. To the best of our knowledge there does not exist a condition that implies the $q$-linear convergence of the function values for the Bregman proximal gradient algorithm. Nevertheless, in the case where $g = 0$, \cref{def:dca_pl} can be shown to incorporate the conditions studied in \cite{bauschke2019linear}. 

In order to make our presentation clearer we first need to define some notions that are standard in the Bregman gradient method analysis. The symmetry coefficient $\alpha(h)$ \cite[Definition 2.3]{bauschke2019linear} measures the lack of symmetry in the generated Bregman divergence and it has the following property:
\begin{equation} \label{eq:symmetry_co}
    \alpha(h)D_h(x, y) \leq D_h(y, x) \leq \alpha(h)^{-1}D_h(x, y) \qquad \forall x, y \in \intr \dom h.
\end{equation}
Moreover, as in \cite{bauschke2019linear} we assume that there exists a lower control function $\theta: \bR_{++} \to \bR_{++}$ with the following property:
\begin{equation} \label{eq:lower_control}
    D_h(\nabla h^*(\nabla h(\bar x) + \lambda \nabla f(\bar x)), \bar x) \geq 
    \theta(\lambda) D_h(\nabla h^*(\nabla h(\bar x)+\nabla f(\bar x)), \bar x).
\end{equation}

The condition that is analyzed in \cite[Definition 3.4]{bauschke2019linear} is the following: We say that the pair $(f,h)$ satisfies a gradient dominated condition if there exist $\eta_1 > 0$ or $\eta_2 > 0$ such that one of the two following inequalities hold:
\begin{align} \label{eq:bregman_pl_1}
    D_h(\nabla h^*(\nabla h(\bar x) + \nabla f(\bar x)), \bar x) &\geq \eta_1 (f(\bar x)-\inf f)
    \\ \label{eq:bregman_pl_2}
    D_h(\bar x, \nabla h^*(\nabla h(\bar x) + \nabla f(\bar x)) & \geq 
    \eta_2 (f(\bar x)-\inf f),
\end{align}
for all $\bar x \in \intr \dom h$. Note that the inequalities stated above cannot directly lead to a linear rate of the algorithm, since the forward step $\nabla h^*(\nabla h(\bar x) + \nabla f(\bar x))$ appearing in the Bregman divergences does not include a stepsize $\lambda > 0$. 
\begin{lemma} \label{lemma:bgm_linear}
    In the setting considered in this subsection, 
    let $\bar x \in \intr \dom h$, $\bar y = \nabla h^*(\nabla h(\bar x) + \tfrac{1}{L}\nabla f(\bar x))$ and $x^+ \in \argmin_{x \in X}g(x)+LD_h(x, \bar y)$. Then, the following hold:
    \begin{lemenum}
        \item \label{lemma:bgm_linear:gaps} The primal and dual gap functions take on the following forms:
        \begin{align*}
        &\gapf(\bar x, \bar y) 
        = g(\bar x) - \min_{x \in X} g(x) + L D_h(x, \bar y) + L D_h(\bar x, \bar y)
        \\
        &\gapg( \bar y, x^+) =
        f(x^+) - \min_{x \in X} L D_h(x, \bar y) + f(x) + L D_h(\bar x, \bar y) = D_{Lh+f}(x^+, \bar x),
        \end{align*}
        \item \label{lemma:bgm_linear:rate} If $g=0$, the gradient dominated inequality \eqref{eq:bregman_pl_1} implies that $F$ satisfies \cref{def:dca_pl} with $\mu_2 = 0$ and $\mu_1 = \tfrac{\alpha(h)\theta(\tfrac{1}{L}) \eta_1}{L}$.
    \end{lemenum}
\end{lemma}
\begin{proof}
    ``\labelcref{lemma:bgm_linear:gaps}'': The form of the primal gap follows by definition. Regarding the equality $\gapg(\bar y, x^+) = D_{Lh-f}(x^+, \bar x)$, it follows using the same algebraic manipulations as in the proof of \cref{lemma:pgm_linear:rate}.

    ``\labelcref{lemma:bgm_linear:rate}'': Starting from \eqref{eq:bregman_pl_1} and utilizing \eqref{eq:lower_control} we obtain:
    \begin{equation*}
        D_h(\nabla h^*(\nabla h(\bar x) + \tfrac{1}{L} \nabla f(\bar x)), \bar x) \geq \theta(\tfrac{1}{L}) \eta_1 (f(\bar x)-\inf f).
    \end{equation*}
    Moreover, using the inequality for the symmetry coefficient \eqref{eq:symmetry_co}, we can further write
    \begin{equation*}
        D_h(\bar x, \nabla h^*(\nabla h(\bar x) + \tfrac{1}{L} \nabla f(\bar x))) \geq \alpha(h)\theta(\tfrac{1}{L}) \eta_1 (f(\bar x)-\inf f).
    \end{equation*}
    Now, note that since $g = 0$ and $\min_x D_h(x, y) = 0$ for any $y \in \intr \dom h$, by removing the dependence on $\bar y$ the primal gap function becomes $\gapf(\bar x) = L D_h(\bar x, \nabla h^*(\nabla h(\bar x)- \tfrac{1}{L}\nabla f(\bar x)))$. Substituting this form in the inequality above we obtain the claimed result.
\end{proof}

\subsection{Linear rate of the anisotropic proximal gradient method}
Consider now the setting of \cref{example:apgm}. 
We remind that $\Phi(x,y)=-\lambda \star h(x-y)$ with $\lambda = 1/L$, while the forward step is $\bar y = \bar x - \tfrac{1}{L}\nabla h^\star(-\nabla f(\bar x))$. Note that this subsumes the Euclidean case covered above. Here we focus on the case where both functions $-f,g$ are strongly convex in the anisotropic sense \cite[Definition 5.8]{laude2022anisotropic}:
\begin{definition}
    Let $\psi \in \Gamma_0(\bR^n)$ be such that $\ran \partial \psi \supseteq \ran \nabla h$. Then we say that $\psi$ is anisotropically strongly convex relative to $h$ with constant $\eta$ if for all $(\bar x, \bar u) \in \gph \partial \psi \cap (\bR^n \times \intr \dom h^*)$ the following inequality holds true
    \begin{equation}
        \psi(x) \geq \psi(\bar x) + \tfrac{1}{\eta} \star h(x-\bar x + \eta^{-1}\nabla h^*(\bar u)) - \tfrac{1}{\eta}\star h(\eta^{-1}\nabla h^*(\bar u)), \quad \forall x \in \bR^n.
    \end{equation}
\end{definition}
In the Euclidean setting, when $g$ is strongly convex a standard technique that leads to linear convergence rate is to transfer the strong convexity to $f$ by adding quadratics. Nevertheless, such a manipulation is not applicable in the anisotropic case as can be seen by the definition of anisotropic strong convexity and thus obtaining $q$-linear rates is not trivial. Instead, one can transfer smoothness which happens to be equivalent to going from the DC-primal to the DC-dual problem \cite[Section 6]{laude2022anisotropic}.

Next we gather the results regarding the $q$-linear convergence rate of the algorithm.
\begin{lemma} \label{lemma:apgm_linear}
    In the setting considered in this subsection, 
    let $\bar x \in \bR^n$, $\bar y = \bar x - \tfrac{1}{L}\nabla h^\star(-\nabla f(\bar x))$ and $x^+ \in \argmin_{x \in \bR^n}g(x)+ \tfrac{1}{L} \star h(x-\bar y)$. Then, the following hold:
    \begin{lemenum}
        \item \label{lemma:apgm_linear:gaps} The primal and dual gap functions take on the following forms:
        \begin{align*}
        &\gapf(\bar x, \bar y) 
        = g(\bar x) - \min_{x \in \bR^n} g(x) + \tfrac{1}{L} \star h(x-\bar y) + \tfrac{1}{L} \star h(\tfrac{1}{L}\nabla h^\star(\nabla f(\bar x)))
        \\
        &\gapg(\bar y, x^+) =
        -f(\bar x) + f(x^+) + \tfrac{1}{L} \star h(x^+ - \bar y) - \tfrac{1}{L} \star h(\tfrac{1}{L}\nabla h^\star(\nabla f(\bar x)))
        \end{align*}
        \item \label{lemma:apgm_linear:rate} If $-f$ and $g$ are anisotropically strongly convex with constants $\eta_f$ and $\eta_g$, then $F$ satisfies \cref{def:dca_pl} with $\mu_1 = \tfrac{\eta_f}{L}$ and $\mu_2 = \tfrac{\eta_g}{L}$.
    \end{lemenum}
\end{lemma}
\begin{proof}
    ``\labelcref{lemma:apgm_linear:gaps}'':
    The forms of the gap functions directly follow from their definitions.

    ``\labelcref{lemma:apgm_linear:rate}'':
    First note that the gap function studied in \cite[Equation 5.1]{laude2022anisotropic} is a scaled version of the one considered in this paper. More presicely, for stepsize $\lambda = \tfrac{1}{L}$, it is defined as $\mathcal{G}_F^h (\bar x, L) := L(F(\bar x) - \valf(\bar x, \bar y)) = L\gapf(\bar x, \bar y)$. Therefore, by \cite[Proposition 5.10]{laude2022anisotropic} we have that
    \begin{equation}
        \eta_f (F(\bar x) - \inf F) \leq L\gapf(\bar x, \bar y).
    \end{equation}
    Moreover, by \cite[Proposition 6.8]{laude2022anisotropic} we have that $-g^\Phi$ is anisotropically smooth with constant $L$ and anisotropically strongly convex with constant $\sigma = \frac{1}{L^{-1} + \eta_g^{-1}}$. Again invoking \cite[Proposition 5.10]{laude2022anisotropic} for the DC-dual problem \cref{eq:dual} we obtain since by \cite[Proposition 6.7]{laude2022anisotropic} $f^\Phi$ is convex, proper and lsc that
    \begin{equation*}
        L\gapg(\bar y, x^+) \geq \sigma(G(\bar y) - \inf G).
    \end{equation*}
    Moreover, we have that $G(\bar y) = F(x^+) + \gapg(\bar y, x^+)$ and thus the inequality above becomes
    \begin{equation*}
        (1-\tfrac{\sigma}{L})\gapg(\bar y, x^+) \geq \tfrac{\sigma}{L}(F(x^+)-\inf F),
    \end{equation*}
    which after some simple algebra implies that
    \begin{equation*}
        \gapg(\bar y, x^+) \geq \tfrac{\eta_g}{L}(F(x^+)-\inf F),
    \end{equation*}
    Summing the two inequalities yields now yields the claimed result.
\end{proof}
The previous result improves upon \cite[Corollary 6.11]{laude2022anisotropic}, where only an $r$-linear convergence rate was shown when $g$ is anisotropically strongly convex.

\section{Conclusion}
\label{sec:conclusion}
We have introduced the $\Phi$-DCA, an algorithmic framework that unifies forward-backward splitting methods based on the notion of generalized convexity, and studied its convergence properties. Our framework allows for a simplified, yet tight analysis and leads to new insights into popular methods that are widely used in practice. In future work we plan to address the following issues: extend the framework to the stochastic setup, introduce a general $\Phi$-subgradient descent scheme akin to the classical convex subgradient descent method and study the guaranteeing saddle point avoidance, in the spirit of \cite{lee2019first}.

\printbibliography
\end{document}